\documentclass[11pt]{amsart}
\usepackage{amsmath}
\usepackage{amssymb, amsthm, amsfonts, amscd, MnSymbol, url}

\usepackage[top=1in, bottom=1in, outer=1.25in, inner=1.25in, heightrounded, marginparwidth=1.15in, marginparsep=.13cm]{geometry}
\usepackage{amsxtra}
\usepackage{epsfig}
\usepackage{verbatim}
\usepackage{graphicx}
\usepackage[all]{xypic}
\usepackage{todonotes}
\usepackage{marginnote}
\usepackage{color}

\usepackage{hyperref}

\input xy
\xyoption{all}

\def\bs{\mathbb{s}}

\def\bA{\mathbb{A}}

\def\bC{\mathbb{C}}

\def\bF{\mathbb{F}}
\def\bG{\mathbb{G}}

\def\bI{\mathbb{I}}

\def\bN{\mathbb{N}}

\def\bQ{\mathbb{Q}}
\def\bR{\mathbb{R}}

\def\bW{\mathbb{W}}

\def\bZ{\mathbb{Z}}
\def\cA{\mathcal{A}}
\def\cB{\mathcal{B}}

\def\cE{\mathcal{E}}

\def\cG{\mathcal{G}}
\def\cH{\mathcal{H}}

\def\cL{\mathcal{L}}
\def\cM{\mathcal{M}}

\def\cO{\mathcal{O}}

\def\cR{\mathcal{R}}
\def\cS{\mathcal{S}}
\def\cT{\mathcal{T}}
\def\cU{\mathcal{U}}

\def\fc{\mathfrak{c}}

\def\fh{\mathfrak{h}}

\def\fL{\mathfrak{L}}

\def\fP{\mathfrak{P}}

\def\End{\operatorname{End}}

\def\GL{\operatorname{GL}}

\def\Gal{\operatorname{Gal}}

\def\Hom{\operatorname{Hom}}

\def\Ind{\operatorname{Ind}}

\def\Isom{\operatorname{Isom}}

\def\Lie{\operatorname{Lie}}
\def\Mat{\operatorname{Mat}}

\def\Spec{\operatorname{Spec}}

\def\Sym{\operatorname{Sym}}

\def\det{\operatorname{det}}
\def\diag{\operatorname{diag}}

\def\dim{\operatorname{dim}}

\def\ker{\operatorname{ker}}

\def\mod{\operatorname{mod}}

\def\univ{\operatorname{univ}}

\def\bs{\backslash}

\def\eps{\varepsilon}

\newcommand{\beas}{\begin{eqnarray*}}
\newcommand{\eeas}{\end{eqnarray*}}

\newcommand{\wt}[1]{\widetilde{#1}}
\def\Levi{\mathrm{Levi}}
\def\Levin{H}
\newcommand{\Isoms}{\underline{\Isom}}

\newcommand{\fpb}{\overline{{\mathbb F}}_p}

\newcommand{\ellcan}{\ell_{\mathrm{can}}}
\newcommand{\uo}{\underline{\Omega}}

\newcommand{\et}{\mbox{\it{\'et}}}
\newcommand{\isomto}{\overset{\sim}{\rightarrow}}
\newcommand{\dual}{^\vee}
\newcommand{\IC}{\bC}
\newcommand{\IQ}{\bQ}
\newcommand{\IR}{\bR}
\newcommand{\ZZ}{\bZ}

\newcommand{\cmfield}{K}
\newcommand{\OK}{\mathcal{O}_\cmfield}
\newcommand{\similitude}{\nu}
\newcommand{\cpct}{\mathcal{U}}
\newcommand{\centralizer}{C}
\newcommand{\imquad}{F}

\newcommand{\Ekap}{\cE_{\kappa}}

\newcommand{\EAR}{\cE_{\uA/S}}
\newcommand{\EARp}{\cE_{\uA/R'}}

\newcommand{\lambdaiso}{\ell}

\newcommand{\modulispace}{M}
\newcommand{\uA}{\underline{A}}
\newcommand{\polarization}{\lambda}
\newcommand{\ENDO}{i}
\newcommand{\level}{\alpha}
\newcommand{\PELtuple}{\left(A, \ENDO, \polarization, \level\right)}
\newcommand{\lattice}{\cL}
\newcommand{\reflex}{E}
\newcommand{\Oe}{\cO_\reflex}
\newcommand{\Oep}{\cO_{\reflex,(p)}}
\newcommand{\Auniv}{\cA}

\newcommand{\siga}{{a_+}}
\newcommand{\sigb}{{a_-}}
\newcommand{\sigp}{\siga}
\newcommand{\sigm}{\sigb}
\newcommand{\sigpm}{{a_{\pm}}}

\newcommand{\Sord}{\cS^{\rm ord} }
\newcommand{\Sordprime}{\cS^{'\rm ord} }
\newcommand{\Ring}{\cR}
\newcommand{\Igusa}{\Ig}
\newcommand{\Witt}{\mathbb{W}}
\newcommand{\loc}{{\rm loc}}
\newcommand{\Sh}{{\rm Sh}}
\newcommand{\ST}{Serre-Tate }
\newcommand{\zz}{{\mathbb Z}}

\newcommand{\arch}{\Sigma} 
\newcommand{\archK}{\Sigma_K} 

\newcommand{\HH}{\cH} 
\newcommand{\transp}[1]{{^t#1}} 
\newcommand{\EA}{\cE_{\uA}} 
\newcommand{\EAo}{\cE_{\uA, \rho}} 
\newcommand{\EAoz}[1]{\cE_{\uA_{#1}, \rho}}
\newcommand{\EAz}[1]{\cE_{\uA_{#1}}}
\newcommand{\rep}{\rho} 
\newcommand{\taut}{{\tilde \tau}}
\newcommand{\sigmat}{{\tilde \sigma}}

\newcommand{\ov}{\overline}
\newcommand{\ul}{\underline}
\newcommand{\ra}{\rightarrow}

\newcommand{\<}{\left\langle}
\renewcommand{\>}{\right\rangle}

\newcommand{\toisom}{\buildrel\sim\over\to}

\newcommand{\Mord}{\cM^{\mathrm{ord}}}

\newcommand{\Ig}{\mathfrak{Ig}^\mathrm{ord}}
\newcommand{\Ign}{\mathrm{Ig}^\mathrm{ord}}

\newcommand{\Mbarord}{\overline{\cM}^{\mathrm{ord}}}

\newcommand{\surjects}{\relbar\joinrel\twoheadrightarrow}

\theoremstyle{plain}
\newtheorem{theorem}{Theorem}[section]
\newtheorem{lemma}[theorem]{Lemma}
\newtheorem{corollary}[theorem]{Corollary}
\newtheorem{proposition}[theorem]{Proposition}
\newtheorem{defn}[theorem]{Definition}

\newtheorem{defi}[theorem]{Definition}

\newtheorem{cor}[theorem]{Corollary}

\newtheorem{prop}[theorem]{Proposition}

\newtheorem{definition}[theorem]{Definition}

\theoremstyle{definition}
\newtheorem{remark}[theorem]{Remark}

\usepackage{color}

\title[$p$-adic $q$-expansion principles on unitary Shimura varieties]{$p$-adic $q$-expansion principles on unitary Shimura varieties}
\date{\today}

\author[A. Caraiani]{Ana Caraiani}
\thanks{Ana Caraiani's research is partially supported by NSF Postdoctoral Fellowship DMS-1204465 and NSF Grant DMS-1501064.}
\author[E. Eischen]{Ellen Eischen}
\thanks{Ellen Eischen's research is partially supported by NSF Grant DMS-1249384.}
\author[J. Fintzen]{Jessica Fintzen}
\thanks{Jessica Fintzen's research is partially supported by the Studienstiftung des deutschen Volkes.}
\author[E. Mantovan]{Elena Mantovan}
\thanks{Elena Mantovan's research is partially supported by NSF Grant DMS-1001077.}
\author[I. Varma]{Ila Varma}
\thanks{Ila Varma's research is partially supported by a National Defense Science and Engineering Fellowship.}

\address{Ana Caraiani\\
 Department of Mathematics\\
 Princeton University\\
 Fine Hall, Washington Road\\
 Princeton, NJ 08544-1000\\
 USA}
 \email{caraiani@princeton.edu}

\address{Ellen Eischen\\
Department of Mathematics\\
University of  Oregon\\
Fenton Hall\\
Eugene, OR 97403\\
USA}
\email{eeischen@uoregon.edu}

\address{Jessica Fintzen\\
 Department of Mathematics\\
 Harvard University\\
 One Oxford Street\\
Cambridge, MA 02138\\
 USA}
 \email{fintzen@math.harvard.edu}

\address{Elena Mantovan\\
 Department of Mathematics\\
CalTech\\
Pasadena, CA 91125\\
 USA}
 \email{mantovanelena@gmail.com}

\address{Ila Varma\\
 Department of Mathematics\\
 Princeton University\\
 Fine Hall, Washington Road\\
 Princeton, NJ 08544-1000\\
 USA}
 \email{ivarma@math.princeton.edu}

\begin{document}
\bibliographystyle{amsalpha}

\newpage
\setcounter{page}{1}
\maketitle
\begin{abstract}
We formulate and prove certain vanishing theorems for $p$-adic automorphic forms on unitary groups of arbitrary signature.  The $p$-adic $q$-expansion principle for $p$-adic modular forms on the Igusa tower says that if the coefficients of (sufficiently many of) the $q$-expansions of a $p$-adic modular form $f$ are zero, then $f$ vanishes everywhere on the Igusa tower.  There is no $p$-adic $q$-expansion principle for unitary groups of arbitrary signature in the literature.  By replacing $q$-expansions with Serre-Tate expansions (expansions in terms of Serre-Tate deformation coordinates) and replacing modular forms with automorphic forms on unitary groups of arbitrary signature, we prove an analogue of the $p$-adic $q$-expansion principle.  More precisely, we show that if the coefficients of (sufficiently many of) the Serre-Tate expansions of a $p$-adic automorphic form $f$ on the Igusa tower (over a unitary Shimura variety) are zero, then $f$ vanishes identically on the Igusa tower.

This paper also contains a substantial expository component.  In particular, the expository component serves as a complement to Hida's extensive work on $p$-adic automorphic forms.

\end{abstract}

\tableofcontents

\section{Introduction}
The purpose of this paper is twofold: to provide an expository guide to the theory of $p$-adic automorphic forms on unitary groups and to formulate and prove certain vanishing theorems for these $p$-adic automorphic forms, which are analogous to the $p$-adic $q$-expansion principle for modular forms. 

In the case of modular forms, which are automorphic forms for $GL_2/\mathbb{Q}$, the $q$-expansion principle is important for constructing families of $p$-adic modular forms and for explicitly computing the Hecke operators acting on ($p$-adic) modular forms. In turn, the algebraic $q$-expansion principle relies on the geometric interpretation of ($p$-adic) modular forms and on the underlying geometry of the moduli spaces they live on. 

Automorphic forms for $GL_n$, when $n>2$, do not have a natural interpretation in terms of algebraic geometry, because the locally symmetric spaces of $GL_n$ do not have the structure of algebraic varieties. The locally symmetric spaces for unitary groups, however, do have the structure of Shimura varieties, and their cohomology realizes systems of Hecke eigenvalues coming from $GL_n$ (either directly since unitary groups are outer forms of $GL_n$, or through congruences - via $p$-adic interpolation). This is why unitary groups have been key in trying to extend results in the Langlands program from $GL_2$ to $GL_n$ in recent years~\cite{shin, HLTT, scholze}. This is also why unitary groups provide a natural context in which to define and study $p$-adic automorphic forms geometrically.

The first part of our paper discusses unitary Shimura varieties, their moduli interpretation and the geometry of their integral models. This leads to the geometric definition of $p$-adic automorphic forms on unitary groups. This is a vast area of research and many different aspects could be highlighted, but we focus on providing an expository account of H. Hida's extensive work in this area, including \cite{hida, hidairreducibility}. In the second part of our paper, we formulate and prove certain analogues of the $q$-expansion principle in this context. We expect that these vanishing theorems will play a key role in constructing families of $p$-adic automorphic forms on unitary groups of arbitrary signature. We discuss these types of theorems in more depth below.

\subsection{Vanishing theorems}

\subsubsection{$q$-expansion principles for modular (and Hilbert modular) forms}

We start with some overview and motivation. We then review the different incarnations of the $q$-expansion principle for modular forms.  (Note that $q$-expansion principles - and the Serre-Tate expansion principle discussed later in this paper - are instances of the principle of analytic continuation, which says an analytic function on a connected domain is completely determined by its restriction to any non-empty open subset and, in particular, is determined by its Taylor expansion around any point.)

Let $\mathcal{H}=\{z\in \IC|\mathrm{Im}\ z>0\}$ be the complex upper half plane. The upper half plane can be identified with the \emph{symmetric space} for the group $SL_2/\mathbb{Q}$: \[\mathcal{H}\simeq SL_2(\mathbb{R})/SO_2(\mathbb{R})\] and it has a natural action of $SL_2(\mathbb{Z})$ by M\"obius transformations, which is equivariant for this identification. Given a congruence subgroup of $SL_2(\mathbb{Z})$, such as \[\Gamma_1(N):=\left\{g\in SL_2(\mathbb{Z})|g\equiv\begin{pmatrix}1&*\\0&1\end{pmatrix}\pmod{N}\right\},\] we can form the associated \emph{locally symmetric space} \[Y_1(N):=\mathcal{H}/\Gamma_1(N).\] This construction generalizes from $SL_2$ to any reductive group over $\mathbb{Q}$, such as $GL_n$ or a unitary group. Moreover, the Betti cohomology of the associated locally symmetric spaces (thought of simply as real manifolds) can be related to automorphic representations of the reductive group. 

In the case of $SL_2$, something special happens: $Y_1(N)$ is not merely a real manifold, but it has a natural complex structure, inherited from the complex structure on $\cH$. A modular form of weight $k$ and level $N$ is a holomorphic function on $\cH$ satisfying certain symmetries under the action of $\Gamma_1(N)$ by M\"obius transformations: \[f\left(\frac{az+b}{cz+d}\right)=(cz+d)^kf(z) \] and also satisfying a growth condition. Since $\begin{pmatrix}1&1\\0&1\end{pmatrix}\in \Gamma_1(N)$, we have \[f(z+1)=f(z),\] which means that $f$ has a Fourier expansion, which ends up looking like \[f(z)=\sum_{n=0}^\infty a_nq^n,\ \mathrm{where}\ q=e^{2\pi iz}.\] We can think of the $q$-expansion as the Taylor expansion of $f$ around the missing point $z=i\infty$ of $Y_1(N)$. The \emph{analytic $q$-expansion principle} says that the Fourier expansion uniquely determines the modular form $f$.

There is also an algebraic $q$-expansion principle, which comes from the fact that modular forms have an interpretation in terms of algebraic geometry. Again, this is something special for $SL_2$ (and other groups that admit Shimura varieties): the Riemann surface $Y_1(N)$ has a natural moduli interpretation (parametrizing elliptic curves) and therefore comes from an algebraic curve defined over $\mathbb{Q}$. The symmetries that the holomorphic function $f$ is required to satisfy make $f$ into a section of a line bundle $\omega^k$ on this curve $Y_1(N)$. This curve is not projective: it will miss a finite number of cusps, one of which corresponds to the point $i\infty$ on the compactification of $\mathcal{H}/\Gamma_1(N)$ as a Riemann surface. If we call this cusp $\infty$, then the coordinate $q$ in the Fourier expansion of $f$ can be identified with a canonical coordinate in a formal neighborhood of $\infty$. The \emph{algebraic} $q$-expansion of $f$ can be identified with the localization of $f$ at the cusp $\infty$.  (The line bundle $\omega^k$ is also canonically trivialized.) The \emph{algebraic $q$-expansion principle} says that a modular form of weight $k$ is uniquely determined by its $q$-expansion. This principle follows from the fact that a section of a line bundle which vanishes in a formal neighborhood of a point on an irreducible curve must vanish everywhere on that curve.

The $p$-adic interpolation of modular forms is crucial for understanding their connection with Galois representations: for example, constructing Galois representations in weight $1$ by congruences and proving modularity via the Taylor-Wiles patching method. 
This leads to the natural question of how to formalize the notion of $p$-adic interpolation and how to define a $p$-adic modular form.

One option is geometric: it only works for groups that have Shimura varieties and uses the geometry of their integral models. In the case of modular forms, this approach goes back to N. Katz~\cite{ka2,kaCM}: $p$-adic modular forms can be thought of as sections of the trivial line bundle over the \emph{Igusa tower}, which can be constructed over the \emph{ordinary locus}. The Igusa tower has the property that it simultaneously trivializes all the line bundles $\omega^k$, corresponding to different weight modular forms $k$. In a very rough sense, this construction can be thought of as a $p$-adic analogue of the upper half plane $\mathcal{H}$. The advantage of this approach is that the underlying geometry provides more tools for studying $p$-adic modular forms (for example, the \emph{Hasse invariant} is such a tool) and answering questions such as when a $p$-adic modular form is classical. 

The $p$-adic $q$-expansion principle says that a $p$-adic modular form is uniquely determined by its $q$-expansion.  This principle relies on the irreducibility of the Igusa tower.  This principle has been extremely important for further studying $p$-adic modular forms. Applications to the construction of $p$-adic $L$-functions are mentioned in Section~\ref{antapp-section}. The $p$-adic $q$-expansion principle is also a crucial ingredient in the work of Buzzard and Taylor on the icosahedral Artin conjecture~\cite{buzzardtaylor} and in generalizations of this type of argument to Hilbert modular varieties. A key aspect of this application is the fact that, for $GL_2$, $q$-expansions of Hecke eigenforms are closely related to Hecke eigenvalues and, therefore, to Galois representations. This is a connection that is not yet understood for other groups, such as unitary groups.

\subsubsection{Principles for other groups, including unitary groups}

In the case of unitary groups or symplectic groups, which admit Shimura varieties, the story described above largely generalizes. Their locally symmetric spaces have an algebraic structure, admit a moduli interpretation, and have integral models. We describe these models in Section~\ref{shvars-section}. One can define automorphic forms in a way similar to how one defines modular forms, and they have an algebro-geometric interpretation as sections of certain vector bundles. We give more details on this in Section~\ref{classautforms-section}. It is also possible to talk about $p$-adic automorphic forms by constructing a (higher-dimensional) Igusa tower over the ordinary locus. These notions are made precise in Section~\ref{padic-section}. However, it is not clear what the best analogue of the $p$-adic $q$-expansion principle would be, in this level of generality. 

There are $q$-expansion principles, or partial results in this direction, in a number of cases.  For Siegel modular forms, i.e. automorphic forms on symplectic groups, there is an algebraic $q$-expansion principle in \cite{CF}.  By \cite[Corollary 8.17]{hida}, there is a $p$-adic $q$-expansion principle for Siegel modular forms.  By \cite[Proposition 7.1.2.14]{lan}, there is an algebraic $q$-expansion principle for scalar-valued automorphic forms on unitary groups of signature $(a, a)$ for any positive integer $a$.  As mentioned in the last paragraph of \cite{hida}, there is a $p$-adic $q$-expansion principle for automorphic forms on unitary groups of signature $(n,n)$. 
The proofs of all of these $q$-expansion principles rely on the existence of cusps, whose formal neighborhoods have canonical coordinates and on the irreducibility of the underlying moduli space (i.e. a Shimura variety or Igusa tower).

For automorphic forms on unitary groups of signature $(a, b)$ with $a\neq b$, the underlying geometry of the associated moduli spaces prevents the existence of a $q$-expansion principle, because these spaces have no cusps (whose formal neighborhoods have canonical coordinates).  In the case of automorphic forms on unitary groups of signature $(a, b)$, when the corresponding Shimura varieties are non-compact, the usual $q$-expansion is replaced by a Fourier-Jacobi expansion, a generalization of the Fourier expansion, in which the coefficients are themselves functions formed from theta-functions.  Nevertheless, there is an algebraic {\it Fourier-Jacobi Principle} for unitary groups \cite[Proposition 7.1.2.14]{lan}.  (This Fourier-Jacobi Principle gives the algebraic $q$-expansion principle for unitary groups of signature $(a, a)$.)

While it is natural to ask for a ``$p$-adic Fourier-Jacobi expansion principle'' for unitary groups of arbitrary signature, a slightly different - but analogous - principle, a ``Serre-Tate expansion principle'' follows more naturally from the existing literature.  The main result of this paper is the formulation and proof of the Serre-Tate expansion principle (in Theorem \ref{STexpprinciple-thm}).  Algebraic $q$-expansions and algebraic Fourier-Jacobi expansions are expansions of a modular (or automorphic) form at the boundary of the Shimura variety.  On the other hand, a Serre-Tate expansion is the expansion of a modular form at an ordinary CM point.  There is a canonical choice of coordinates for the local ring at the ordinary point; these are called {\it Serre-Tate deformation coordinates}.  Roughly speaking our main result (stated precisely in Theorem \ref{STexpprinciple-thm}) says that given suitable conditions on the prime $p$ (namely, when $p$ splits completely in the reflex field), if $f$ is an automorphic form on a unitary group and for each irreducible component $C$ of the associated Igusa tower, a Serre-Tate expansion of $f$ at some CM point in $C$ is $0$, then $f$ vanishes identically on the Igusa tower.  The proof relies on Hida's description of the geometry of the Igusa tower. The key point is, again, the irreducibility of the Igusa tower.

\subsection{Anticipated applications}\label{antapp-section}

As noted above, the use of $q$-expansion principles in the construction of $p$-adic families of modular forms is well-established.  In \cite{kaCM}, Katz 
used the $q$-expansion principle for Hilbert modular forms to study congruences between values of different Hilbert modular forms, which led to the construction of certain $p$-adic families of Hilbert modular forms.  Similarly, in \cite{apptoSHL}, the second author used the $q$-expansion principle to construct $p$-adic families of automorphic forms on unitary groups of signature $(a, a)$ for all positive integers $a$.  Katz's $p$-adic families of Hilbert modular forms are the main ingredient in his construction of $p$-adic $L$-functions for CM fields \cite{kaCM}.  Analogously, the second author constructed the $p$-adic families of automorphic forms in \cite{apptoSHL} to complete a step in the construction of $p$-adic $L$-functions (for unitary groups) proposed in \cite{HLS}. 

We plan to use the Serre-Tate expansion principle in Theorem \ref{STexpprinciple-thm} analogously to how the $q$-expansion principle is used in contexts in which $q$-expansions exist. More precisely, in a joint paper in preparation, we are using the Serre-Tate expansion principle introduced in this paper to construct $p$-adic families of automorphic forms on unitary groups of signature $(a, b)$ with $a\neq b$.  As explained in \cite{emeasurenondefinite}, the lack of such a principle in the literature was an obstacle faced by the second author in her effort to extend her results on $p$-adic families of automorphic forms to unitary groups of arbitrary signature.  This paper eliminates that obstacle and fills in a hole in the literature.  We also are using the expository portion of this paper as part of the foundation for our construction of these families.

One advantage of expansions around CM points over $q$-expansions is that they can be used for compact as well as non-compact Shimura varieties. The Serre-Tate expansion has been used before by Hida (for example, to define his idempotent in~\cite{hida}) and also appears in work of Brooks~\cite{brooks} (for Shimura curves) and Burungale and Hida~\cite{BH} (for Hilbert modular varieties) with applications to special values of $p$-adic $L$-functions.

In a more speculative direction, we note the potential for applications to homotopy theory.  Certain $p$-adic families of modular forms, studied in terms of their $q$-expansions, were used to defined an invariant (the {\it Witten genus}) in homotopy theory \cite{hopkins94, hopkinsICM, AHR}.  The Witten genus is a $p$-adic modular form valued invariant that occurs in the theory of {\it topological modular forms}.  Recently, there have been attempts to construct an analogue of the Witten genus in the theory of {\it topological automorphic forms}, where there is conjecturally an invariant taking values in the space of $p$-adic automorphic forms on unitary groups of signature $(1, n)$ \cite{beh}.  Vanishing theorems analogous to the $q$-expansion principle will likely play an analogously important role in this context.

\subsection{Structure of the paper}\label{paperstructure-section}

We now provide a brief overview of the paper.  Section \ref{shvars-section} introduces Shimura varieties for unitary groups and the associated moduli problem.  We work with these Shimura varieties throughout most of the paper.  Section \ref{classautforms-section} reviews the theory of classical automorphic forms on unitary groups, from several perspectives.  Section \ref{padic-section} introduces Hida's geometric theory of $p$-adic automorphic forms 
 (i.e. over the ordinary locus).  This section includes details about the Igusa tower, as well as the space of $p$-adic automorphic forms (defined as global sections of the structure sheaf over the Igusa tower). Section \ref{ST-section} covers the main results of this paper, namely the {\it \ST expansion principle}, an analogue of the $q$-expansion principle, for $p$-adic automorphic forms on unitary groups of arbitrary signature.   We are using this result in a paper in preparation that constructs families of $p$-adic automorphic forms on unitary groups of arbitrary signature. Finally, Section~\ref{section on pullbacks} discusses how Serre-Tate expansions behave with respect to pullbacks, as an example of the kind of application we have in mind to computational aspects of $p$-adic automorphic forms on unitary groups. 

\subsection{Notation and conventions}\label{notation-section}

We now establish some notation and conventions that we will use throughout the paper.

First, we establish some notation for fields.  Fix a totally real number field $\cmfield^+$ and an imaginary quadratic extension $\imquad$ of $\IQ$.  Define $\cmfield$ to be the composition of $\cmfield^+$ and $\imquad$.  Let $c$ denote complex conjugation on $\cmfield$, i.e. the generator of $\Gal(\cmfield/\cmfield^+)$. We denote by $\arch$ the set of complex embeddings of $\cmfield^+$, and we denote by $\archK$ the set of complex embeddings of $\cmfield$.  We typically use $\tau$ to denote an element of $\arch$, and for each $\tau \in \arch$, we fix an extension $\tilde\tau$ of $\tau$ to $K$, i.e. $\tilde \tau$ is an element of $\archK$.
 A reflex field will be denoted by $\reflex$ (with subscripts to denote different reflex fields when there is more than one reflex field appearing in the same context).  
 Given a local or global field $L$, we denote the ring of integers in $L$ by $\mathcal{O}_L$.  We write $\bA$ to denote the adeles over $\IQ$, we write $\bA^\infty$ to denote the adeles away from the archimedean places, and we write $\bA^{\infty,p}$ to denote the adeles away from the archimedean places and $p$.

Fix a rational prime $p$ that splits as $p=w\cdot w^c$ in the imaginary quadratic extension $\imquad/\mathbb{Q}$. We make this assumption in order to ensure that our unitary group at $p$ is a product of (restrictions of scalars of) general linear groups.  Instead, we could assume that every place of $\cmfield^+$ above $p$ splits in the quadratic extension $\cmfield/\cmfield^+$ and choose a CM type for $\cmfield$. In addition, we restrict our attention to the case when the prime $p$ is \emph{unramified} in $K$. This ensures that the Shimura varieties we consider have smooth integral models over $\Oep$ (where $\Oe$ is the ring of integers in the reflex field $\reflex$ of these Shimura varieties) when no level structure at $p$ is imposed.

To help the reader keep track of each setting, we adhere to the following conventions for fonts used to denote schemes, integral models, and formal completions throughout the paper.  Schemes over $\bQ$ are in normal font, their integral models are in mathcal font, and their formal completions are in mathfrak font.

\section{Unitary Shimura varieties}\label{shvars-section}
In this section, we introduce unitary Shimura varieties.  In Section \ref{unitarygroups-section}, we introduce PEL data and conventions for unitary groups, with which we work throughout the paper.  Section \ref{PELmoduli-section} introduces the PEL moduli problem, and Section \ref{abvc-section} specializes to the setting over $\IC$. In our exposition, we follow \cite{kottwitz} and \cite{lan}.

\subsection{PEL data and unitary groups}\label{unitarygroups-section} The following definition of the PEL datum follows \cite[Section 1.2]{lan}, and it is an integral version of the datum in  \cite[Section~4]{kottwitz}.

By a {\it PEL datum}, we mean a tuple $\left(\cmfield, c, L, \langle, \rangle, h\right)$ consisting of
\begin{itemize}
\item the CM field $\cmfield$ equipped with the involution $c$ introduced in Section \ref{notation-section},
\item an $\OK$-lattice $L$, i.e. a finitely generated free $\ZZ$-module with an action of $\OK$,
\item a non-degenerate Hermitian pairing $\langle\cdot,\cdot\rangle:L\times L\to \mathbb{Z}$ satisfying $\langle k\cdot v_1,v_2 \rangle=\langle v_1,k^c\cdot v_2 \rangle$ for all $v_1, v_2\in L$ and $k\in \OK$, 
\item an $\mathbb{R}$-algebra endomorphism \[h:\mathbb{C}\to \mathrm{End}_{\OK\otimes_\mathbb{Z}\mathbb{R}}(L\otimes_{\mathbb{Z}}\mathbb{R})\] such that $(v_1,v_2)\mapsto \langle v_1,h(i)\cdot v_2\rangle$ is symmetric and positive definite and such that $\langle h(z) v_1, v_2\rangle = \langle v_1, h(\overline{z})v_2 \rangle$.
\end{itemize}
Furthermore, for considering the moduli problem over a $p$-adic ring and for defining $p$-adic automorphic forms, we require
\begin{itemize}
\item{$L_p:=L\otimes_{\mathbb{Z}}\mathbb{Z}_p$ is self-dual under the alternating Hermitian pairing $\langle \cdot,\cdot\rangle_p$ on $L\otimes_{\mathbb{Z}}\mathbb{Q}_p$.} 
\end{itemize}

To the PEL datum $\left(\cmfield, c, L, \langle, \rangle, h\right)$, we associate algebraic groups $GU= GU\left(L, \langle, \rangle\right)$,
$U=U\left(L, \langle, \rangle\right)$, and $SU=SU\left(L, \langle, \rangle\right)$ defined over $\ZZ$, whose $R$-points (for any $\ZZ$-algebra $R$) are given by
\begin{align*}
	GU(R)&:=\left\{(g,\similitude)\in \mathrm{End}_{\OK\otimes_\mathbb{Z}R}(L\otimes_{\mathbb{Z}}R) \times R^\times\mid\langle g\cdot v_1,g\cdot v_2\rangle =\similitude\langle v_1,v_2\rangle\right\}\\
	U(R)&:=\left\{g\in \mathrm{End}_{\OK\otimes_\mathbb{Z}R}(L\otimes_{\mathbb{Z}}R) \mid\langle g\cdot v_1,g\cdot v_2\rangle =\langle v_1,v_2\rangle\right\}\\
	SU(R)&:=\left\{g\in U(R) \mid \det g=1\right\}.
\end{align*} 
Note that $\similitude$ is called a {\it similitude factor}.   In the following, for $R=\bQ$ or $\IR$, we also write 
\[GU_+(R):=\left\{(g,\similitude)\in GU(R)\mid \similitude>0\right\}.\]

Moreover, given a PEL datum $\left(\cmfield, c, L, \langle, \rangle, h\right)$, we define the $\IR$-vector space with an action of $\cmfield$ 
\begin{align*}
V:=L\otimes_{\ZZ}\IR.
\end{align*}
Then $h_\bC=h\times_{\mathbb{R}}\mathbb{C}$ gives rise to a decomposition $V\otimes_{\IR}\mathbb{C}=V_1\oplus V_2$ (where $h(z)\times 1$ acts by $z$ on $V_1$ and by $\bar{z}$ on $V_2$).  We have decompositions $V_1 = \oplus_{\tau\in\archK} V_{1, \tau}$ and $V_2 = \oplus_{\tau\in\archK}V_{2, \tau}$ {induced from the decomposition of $K\otimes_{\bQ} \bC = \oplus_{\tau \in \Sigma_K} \bC$ where only the $\tau$-th $\bC$ acts nontrivially on $V_{1,\tau} \oplus V_{2,\tau}$, acting via the standard action on $V_{1,\tau}$.} As defined in \cite[Definition 1.2.5.2]{lan}, the {\it signature} of $\left(V, \langle, \rangle, h\right)$ is the tuple of pairs $\left(\siga_\tau, \sigb_\tau\right)_{\tau\in\archK}$ such that $\siga_\tau = \dim_\IC V_{1, \tau}$ and $\sigb_\tau = \dim_\IC V_{2, \tau}$ for all $\tau\in\archK$.  Let
\begin{align*}
n=\siga_\tau+\sigb_\tau.
\end{align*}
Note that $n$ is independent of $\tau$ and furthermore $a_{\pm\tau} = a_{\mp\tau^c}$.

In order to define automorphic forms of non scalar weight in Section \ref{classautforms-section} and \ref{padic-section}, we define the algebraic group over $\bZ$ 
$$\Levin :=\prod\limits_{\tau \in \arch} \GL_{\siga_{\tilde\tau}} \times \GL_{\sigb_{\tilde\tau}},$$  where $\wt \tau \in \archK$ is a previously fixed lift of $\tau \in \arch$. Note that $\Levin_\bC$ can be identified with the Levi subgroup of $U({\bC})$ that preserves the decomposition $V_\bC=V_1\oplus V_2$. This identification also works over $\bZ_p$, as we will see in Section {\ref{Igusalevel}}.

Moreover, using the decomposition $\bC \otimes \bC = \bC \oplus \bC$ of $\bR$-modules (where $z \in \bC$ acts on the first summand by $z$ and on the second summand by $\ov z$) we define $ \mu: \bC \ra V_\bC \text{ by } z \mapsto h_\bC(z,1)$.  (Compare \cite[Section~12]{Milne}.)
Then the {\it reflex field} is defined to be the field of definition of the $GU(\bC)$-conjugacy class of $\mu$ (or equivalently as the conjucagy class of $V_1$). Henceforth, we denote the reflex field by $E$  (note that $E\subset \IC$).

\subsection{PEL moduli problem}\label{PELmoduli-section}
The goal of this section is to introduce PEL-type unitary Shimura varieties from a moduli-theoretic perspective. We will restrict our attention to cases where these Shimura varieties have no level structure and good reduction at $p$.  For more details, see~\cite[Section 5]{kottwitz} or \cite[Section 1.4]{lan}.

We now define a moduli problem for abelian varieties equipped with extra structures (more precisely, polarizations, endomorphisms and level structure) and which will be representable by unitary Shimura varieties that have integral models.  For each open compact subgroup $\cpct\subset GU(\mathbb{A}^\infty)$, consider the moduli problem \[(S,s)\mapsto \{\PELtuple\}\] which assigns to every connected, locally noetherian scheme $S$ over $E$ together with a geometric point $s$ of $S$ the set of tuples $\PELtuple$, where
\begin{itemize}
\item $A$ is an abelian variety over $S$ of dimension $g:=[K^+:\mathbb{Q}]\cdot n$,
\item $\ENDO:K\hookrightarrow \mathrm{End}^0(A):=(\mathrm{End}(A))\otimes_{\mathbb{Z}}\mathbb{Q}$ is an embedding of $\mathbb{Q}$-algebras,
\item $\polarization: A\to A^\vee$ (where $A^\vee$ denotes the dual abelian variety) is a polarization satisfying $\lambda \circ i(k^c)=i(k)^\vee \circ \lambda$ for all $k\in K$,
\item ${\level}$ is a $\pi_1(S,s)$-invariant $\cpct$-orbit of $K\otimes_{\mathbb{Q}}\mathbb{A}^\infty$-equivariant isomorphisms \[L\otimes_{\mathbb{Z}}\mathbb{A}^\infty \toisom V_fA_s,\] which takes the Hermitian pairing $\langle \cdot,\cdot \rangle$ on $L$ to an $(\mathbb{A}^\infty)^\times$-multiple of the $\lambda$-Weil pairing on the rational (adelic) Tate module $V_fA_s$.
\end{itemize}
Note that $\mathrm{Lie}\ A$ is a locally free $\cO_S$-module of rank $g$ and has an induced action of $K$ via $i$. The tuple $\PELtuple$ must satisfy Kottwitz's \emph{determinant condition}:  \[\mathrm{det}(K|V_1)=\mathrm{det}_{\cO_S}(K|\mathrm{Lie}\ A).\] Here, by $\mathrm{det}(K|V_1)$ we denote the element in $E[K^\vee]=\Sym(K^\vee)\otimes_\bQ E$, for $K^\vee$ the $\bQ$-vector space dual to $K$, defined by $k\mapsto \mathrm{det}_{\IC}(k|V_1)$, for all $k\in K$. By definition of the reflex field, $\mathrm{det}(K|V_1)\in E[K^\vee]\hookrightarrow \IC[K^\vee]$. Similarly, $\mathrm{det}_{\cO_S}(K|\mathrm{Lie}\ A)$ denotes the element in $\cO_S[K^\vee]=\Sym(K^\vee)\otimes_\bQ \cO_S(S)$
defined by $k\mapsto \mathrm{det}_{\cO_S}(k|\mathrm{Lie}\ A)$, for all $k\in K$. The determinant condition is an equality of elements in $\cO_S[K^\vee]$, after taking the image of $\mathrm{det}(K|V_1)$  under the structure homomorphism of $E$ to $\cO_S(S)$.

Two tuples $\PELtuple$ and $\left(A',\ENDO',\polarization',\level'\right)$ are equivalent if there exists an isogeny $A\to A'$ taking $\ENDO$ to $\ENDO'$, $\polarization$ to a rational multiple of $\polarization'$, and $\level$ to $\level'$. We note that the definition is independent of the choice of geometric point $s$ of $S$. We can extend the definition to non-connected schemes by choosing a geometric point for each connected component.

If the compact open subgroup $\cpct$ is neat (in particular, if it is sufficiently small) as defined in \cite[Definition 1.4.1.8]{lan}, then this moduli problem is representable by a smooth, quasi-projective scheme $\modulispace_\cpct/E$.

From now on, assume that $\cpct=\cpct^p\cpct_p$ is neat and that $\cpct_p\subset GU\left(\mathbb{Q}_p\right)$ is hyperspecial. 
We can construct an integral model of $M_\cpct$ by considering an integral version of the above moduli problem. To a pair $(S,s)$, where $S$ is now a scheme over $\Oep$, we assign the set of tuples $\left(A,\ENDO,\polarization, \level^p\right)$, where
\begin{itemize}
\item $A$ is an abelian variety over $S$ of dimension $g$,
\item $i:{\cO_{K,(p)}} \hookrightarrow (\mathrm{End}(A))\otimes_{\mathbb{Z}}\mathbb{Z}_{(p)}$ is an embedding of $\mathbb{Z}_{(p)}$-algebras
\item $\lambda: A\to A^\vee$ is a prime-to-$p$ polarization satisfying $\lambda \circ i(k^c)=i(k)^\vee \circ \lambda$ for all $k\in \cO_K$,
\item $\alpha^p$ is a $\pi_1(S,s)$-invariant $\cpct^p$-orbit of $K\otimes_{\mathbb{Q}}\mathbb{A}^{\infty,p}$-equivariant isomorphisms \[L\otimes_{\mathbb{Z}}\mathbb{A}^{\infty,p} \toisom V^p_fA_s,\] which takes the Hermitian pairing $\langle \cdot,\cdot \rangle$ on $L$ to an $(\mathbb{A}^{\infty,p})^\times$-multiple of the $\lambda$-Weil pairing on $V^p_fA_s$ (the Tate module away from $p$).
\end{itemize}
In addition, the tuple $\PELtuple$ must satisfy Kottwitz's \emph{determinant condition}: \[\mathrm{det}(\OK|V_1)=\mathrm{det}_{\cO_S}(\OK|\mathrm{Lie}A).\]
Here, the determinant condition is an equality of elements in $\cO_S[\OK^\vee]$, after taking the image of $\mathrm{det}(\OK|V_1)\in (\Oep)[\OK^\vee]$ 
under the structure homomorphism of  $\Oep$ to $\cO_S$, for $\OK^\vee$ the dual $\ZZ$-module of $\OK$. 

Two tuples $\PELtuple$ and $\left(A',\ENDO',\polarization',\level'\right)$ are equivalent if there exists a prime-to-$p$ isogeny $A\to A'$ taking $\ENDO$ to $\ENDO'$, $\polarization$ to a prime-to-$p$ rational multiple of $\polarization'$ and $\level$ to $\level'$.

This moduli problem is representable by a smooth, quasi-projective scheme $\cM_\cpct$ over $\Oep$. (See, for example, page 391 of~\cite{kottwitz} for a discussion of representability and smoothness. The representability is reduced to the Siegel case, proved in~\cite{mumford}, while the smoothness follows from Grothendieck-Messing deformation theory.)  We have a canonical identification \[M_\cpct=\cM_\cpct\times_{\mathrm{Spec}\ (\Oep)}\mathrm{Spec}\ E,\] which can be checked directly on the level of moduli problems.  As the level $\cpct^p$ varies, the inverse system of Shimura varieties $\cM_\cpct$ has a natural action of $GU(\mathbb{A}^{\infty,p})$. (More precisely, $g\in GU(\mathbb{A}^{\infty, p})$ acts by precomposing the level structure $\alpha$ with it.) Since our interest is in the $p$-adic theory, we will fix and suppress the level $\cpct$ starting from Section \ref{padic-section}.

\subsection{Abelian varieties and Shimura varieties over $\IC$}\label{abvc-section}
In this section, we specialize to working over $\IC$.  Our goal is to sketch how the set $M_\cpct (\mathbb{C})$ of complex points of $M_\cpct$ is naturally identified with the set of points of a finite union of locally symmetric complex varieties corresponding to $(GU,h)$.  (For more details, see~\cite[Section 8]{kottwitz}.)

We remark that what we show is merely a bijection of sets. Proving that the Shimura varieties corresponding to $(GU,h)$ are moduli spaces of abelian varieties over $\IC$ would also require matching the complex structures on the two sides.

\subsubsection{Abelian varieties over $\IC$}
Recall that the $\bC$-points of an abelian variety $A/\IC$ are of the form $V(A)/\Lambda$, where $\Lambda$ is a $\bZ$-lattice in a complex vector space $V(A)$. Any abelian variety over $\bC$ admits a polarization; since $V(A)/\Lambda$ comes from a complex abelian variety $A$, it is also polarizable, i.e. there exists a nondegenerate, positive definite Hermitian form
 	$$\lambda_{\bC}: V(A)\times V(A) \rightarrow \bC \quad \mbox{s.t. } \lambda_{\bC}(\Lambda,\Lambda) \subset \bZ.$$
 	We call each such Hermitian form $\lambda_{\bC}$ a {\it polarization} of $V(A)/\Lambda$. It may be better to think of a polarization as an alternating form $\lambda_{\bR}: V(A) \times V(A) \rightarrow \bR$ satisfying $\lambda_\bR(iu,iv) = \lambda_\bR(u,v)$ for all $u,v \in V(A)$ and
	$$\lambda_{\bC}(u,v) = \lambda_{\bR}(u,iv) + i\lambda_{\bR}(u,v).$$
It is enough to characterize a pair $(A,\lambda_{\bC})$, where $A$ is an abelian variety of dimension $g$, by considering the following triple:
\begin{enumerate}
\item the free $\bZ$-module $\Lambda = H_1(A,\bZ)$ of rank $2g$
\item the $\bR$-algebra homomorphism $\bC \rightarrow \End_{\bR}(\Lambda \otimes \bR) = \End_{\bR}(H_1(A,\bR)) = \End_{\bR}(\Lie A)$ describing the complex structure on $\Lie A$ (so $V(A):=\Lambda\otimes \bR$, endowed with this complex structure)
\item the alternating form on $\Lambda=H_1(A,\bZ)$ induced by $\lambda_{\bC}$ denoted by $\langle \cdot,\cdot\rangle$ after identifying $A(\bC) \cong \Lie A(\bC)/H_1(A,\bZ)$
\end{enumerate} 

\subsubsection{Shimura varieties over $\IC$}\label{moduli-section}
Recall that, associated to the PEL datum, we have the  $\IR$-vector space $V=L\otimes_\bZ\IR$, which is endowed with an action of $K$  and the complex structure defined by $h$. (Note that the complex structure depends uniquely on $h(i)\in \End_{K\otimes_\bQ\IR}(V)$.)

Let $\fh$  denote the set of elements $I\in \End_{K\otimes_\bQ\IR}(V)$ which satisfy
\begin{enumerate}
\item $I^2 = -1$  
\item $I^{c} = -I$
\item $(w,v) \mapsto \<w,Iv\>$ is a positive  or negative definite form on $V$
\item the $K\otimes_\bQ\IC$-structures on $V$ defined by $I$ and $h(i)$ are isomorphic.
\end{enumerate}
In \cite[Lemmas 4.1 and 4.2]{kottwitz}, Kottwitz shows that the set $\fh$ is equal to $ GU(\bR)/\centralizer_{h}$, for $\centralizer_h$ the stabilizer of $h(i)$ in $GU(\bR)$, and that it can be identified with a finite union of copies of the symmetric domain for the identity component of $GU(\IR)$. 

For $\cpct \subset GU(\bA^\infty)$  a neat open compact subgroup, we define the quotient
	\begin{equation}X_{\cpct} = GU(\bQ)\bs (GU(\bA^\infty)/\cpct \times \fh).\end{equation}
We sketch how the set of complex points of $M_\cpct$ corresponds to a disjoint union of finitely many copies of $X_{\cpct}$. 

By definition, the set  $M_{\cpct}(\bC) $  parametrizes equivalence classes of tuples $\PELtuple$ where
\begin{enumerate}
\item $A$ is an abelian variety over $\bC$,
\item $\ENDO: K \hookrightarrow \End^0(A)$ is an embedding of $\bQ$-algebras,
\item $\polarization: A \rightarrow A^{\vee}$ is a polarization satisfying $\polarization \circ i(b^c) = i(b)^{\vee} \circ\polarization$ for all $b \in K$,
\item $\level$ 
is a $\cpct$-orbit of isomorphisms of skew-Hermitian $K$-vector spaces 
(in the sense of Kottwitz, i.e. preserving the pairing only up to scalar) $L \otimes_{\bZ} \bA^\infty \cong H_1(A,\bA^\infty)$.
\end{enumerate}
In addition,  the tuple $\PELtuple$ must satisfy the determinant condition.

Every equivalence class of tuples $\PELtuple$ satisfying the above restrictions gives rise to an element $GU(\bA^\infty)/\cpct \times \fh$ as follows. 
The existence of the equivalence class of isomorphisms $\level$ implies that the skew-Hermitian $K$-vector spaces  $L_\bQ:=L\otimes_\ZZ\bQ$ and $H = H_1(A,\bQ)$ are isomorphic over any finite place of $\bQ$. We conclude in particular that $H_1(A,\bQ)$ and $L_\bQ$ have the same dimension over $K$. 
By \cite[Lemma 4.2]{kottwitz}, $H_\IR$ and $V=L_\bQ \otimes_\bQ \bR$ are isomorphic as skew-Hermitian $K$-vector spaces  if and only if they are isomorphic as $K\otimes_\bQ\IC$-modules.   The  natural complex structure on $H_\IR\cong {\rm Lie}(A)$ gives a decomposition of $H_\IC$ as $H_1\oplus H_2$, and the determinant condition implies that $H_1$ is isomorphic to $V_1$ as $K\otimes_\bQ\IC$-modules. Thus, in order to deduce that $H_\IR$ and $V$ are isomorphic as $K\otimes_\bQ\IC$-modules it suffices to prove that $V_2$ and $H_2$ are isomorphic as $K\otimes_\bQ\IC$-modules. Note that if we denote by $W_\tau$   the $\IC$-subspace where $K$ acts via $\tau$, for $\tau:K\hookrightarrow \IC$ a complex embedding ($\tau\in\Sigma_K$) and  $W$ any $K\otimes_\bQ\IC$-module, then two $K\otimes_\bQ\IC$-modules $W,W'$ are isomorphic if and only if $\dim_\IC W_\tau=\dim_\IC W'_\tau$, for all $ \tau\in\Sigma_K$.
Let $\tau\in\Sigma_K$. The determinant condition implies that $\dim_\IC V_{1,\tau}=\dim_\IC H_{1,\tau}$, the decompositions $V_\IC=V_1\oplus V_2$ and $H_\IC= H_1\oplus H_2$ imply
$\dim_\IC V_{\tau}+\dim_\IC V_{\tau^c}= \frac{1}{2}(\dim_\IC V_{1,\tau}+\dim_\IC V_{2,\tau^c}+\dim_\IC V_{1,\tau^c}+\dim_\IC V_{2,\tau})=\dim_\IC V_{1,\tau}+\dim_\IC V_{2,\tau},$ and  $\dim_\IC H_{\tau}+\dim_\IC H_{\tau^c}= \dim_\IC H_{1,\tau}+\dim_\IC H_{2,\tau},$ and the equality  $\dim_K L_\bQ=\dim_K H$ implies $\dim_\IC V_{\tau}+\dim_\IC V_{\tau^c}=\dim_\IC H_{\tau}+\dim_\IC H_{\tau^c}$. We deduce that $\dim_\IC V_{2,\tau}=\dim_\IC H_{2,\tau}$ for all $\tau\in\Sigma_K$, i.e. that $H_2$ and $V_2$ are also isomorphic as $K\otimes_\bQ\IC$-modules.   We conclude that the skew-Hermitian $K$-vector spaces $H$ and $L_\bQ$ are isomorphic over any place $v$ of $\bQ$.

When the Hasse principle holds, this implies the existence of  an isomorphism of skew-Hermitian $K$-vector spaces between $H$ and $L_\bQ$.   In general, there are $$\left|\ker^1(\bQ,GU):=\ker\left(H^1(\bQ,GU)\ra \prod\limits_{v \text{ place of } \bQ} H^1(\bQ_v,GU) \right)\right|$$ isomorphism classes $L^{(i)}$ of skew-Hermitian $K$-vector spaces isomorphic to $L_\bQ$ at every place of $\bQ$ (and let $L^{(1)}=L_\bQ$). 

Let $1\leq i \leq |\ker^1(\bQ,GU)|$. We define $GU^{(i)}$ to be the unitary similitude group over $\bQ$ defined by $L^{(i)}$ (so in particular, $GU^{(1)} = GU$). Choose local isomorphisms $L_{\bQ_v}\cong L^{(i)}_{\bQ_v}$ for all places $v$ of $\bQ$, and let  $GU^{(i)}(\bQ)$ act on $GU(\bA^\infty)/\cpct\times\fh$ via
the induced isomorphisms $GU^{(i)}_{\bQ_v}\cong GU_{\bQ_v}$.  We define $X^{(i)}_{\cpct} = GU^{(i)}(\bQ)\bs (GU(\bA^\infty)/\cpct \times \fh)$, and we define $M^{(i)}_\cpct(\IC)$ to be the subset of $M_\cpct(\IC)$ parameterizing tuples such that $H$ is isomorphic to $L^{(i)}$. Thus,   $M_\cpct(\IC)=\coprod_{i} M^{(i)}_\cpct(\IC)$. We show that $M^{(i)}_\cpct(\IC)$ naturally identifies with $X^{(i)}_{\cpct} $.

Let $i$, where $1\leq i \leq |\ker^1(\bQ,GU)|$, be such that $H$ and $L^{(i)}$ are isomorphic skew-Hermitian $K$-vector spaces, and  choose an automorphism $\level_{\bQ}: H\isomto L^{(i)}$. Then, the automorphism $(\level_{\bQ} \otimes \bI_{\bA^\infty}) \circ \level$ of $L^{(i)} \otimes_\bQ \bA^\infty$ defines an element of $GU(\bA^\infty)$, but since $\level$ is only well-defined up to its orbit in $\cpct$, such an isomorphism determines an element of $GU(\bA^\infty)/\cpct$.  Under $\level_\bQ$, the complex structure on $H_1(A,\bR)$ defines a complex structure on $V$, which is conjugate to $h$ by an element in $GU(\IR)$, i.e., an element in $\fh$.

Therefore, each class of tuples $\PELtuple$ along with a choice of isomorphism $\level_{\bQ}$ determines an element of $GU(\bA^\infty)/\cpct \times \fh$. Forgetting the isomorphism $\level_{\bQ}$ is equivalent to taking the quotient by the left action of $GU^{(i)}(\bQ)$. Thus, to each point of $M^{(i)}_\cpct(\IC)$ we associated a point on $X^{(i)}_\cpct$, and this map is in fact a bijection. 

Note that in \cite[Sections 7 and 8]{kottwitz} Kottwitz shows that under our assumptions (case A in loc. cit.) if $n$ is even the Hasse principle holds, and if $n$ is odd the natural map $\ker^1(\bQ,Z)\to \ker^1(\bQ,GU)$, for $Z$ the center of $GU$, is a bijection, and furthermore that the subvarieties $M^{(i)}_\cpct(\IC)$  are all isomorphic to $M^{(1)}_\cpct(\IC) = X_{\cpct}$. 

For the later sections, we will denote a connected component of $M_{\cpct}(\bC)$ (or equivalently, of $X_{\cpct}(\bC)$) as $S_{\cpct}(\bC)$. Note that any two connected components are isomorphic as complex manifolds.

\section{Classical automorphic forms}\label{classautforms-section}

In this section we will first recall the classical definition of automorphic forms on unitary groups over $\bC$ following \cite{Shimura}, and then describe equivalent viewpoints that let us generalize to work over base rings other than $\bC$.

\subsection{Classical definition of complex automorphic forms on unitary groups.}\label{autformsC-section}
For the moment, suppose $a_{+\tilde\tau}a_{-\tilde\tau}\neq 0$ for all $\tau \in \arch$.  Consider the domain $\HH$ for $GU_+(\bR):$ 
\begin{equation*}
	\HH= \prod_{\tau \in \arch} \HH_{\siga_\taut \times \sigb_\taut} \mbox{ with }  \HH_{\siga_\taut \times \sigb_\taut} = \{ z \in \Mat_{\sigb_\taut \times \siga_\taut}(\bC) \, | \, 1-{}^tz^cz \mbox{ is positive definite} \}.
\end{equation*}
Note that $\Isom_{\cO_K\otimes_\bZ \bR}(L \otimes_\bZ \bR) \simeq \GL_n(\cO_K \otimes \bR) \simeq \GL_n(\prod_{\tau \in \Sigma}\bC) \simeq \prod_{\tau \in \Sigma} \GL_n(\bC)$. We use this identification to write $g \in GU_+(\bR)$ as
$\left(\begin{smallmatrix} a_{g,\tau} & b_{g,\tau} \\ c_{g,\tau} & d_{g,\tau} \end{smallmatrix}\right)_{\tau \in \Sigma} \in GU_{+}(\bR)$, where  $a_{g,\tau} \in \GL_{\siga_\taut}(\bC)$ and $d_{g,\tau} \in \GL_{\sigb_\taut}(\bC)$.  Then the action of $g$ on $\HH$ is given by
\begin{equation*}
	gz = ((a_{g,\tau}z_{\tau}+b_{g,\tau})(c_{g,\tau}z_{\tau}+d_{g,\tau})^{-1})_{\tau \in \Sigma} \, \mbox{ for } \, z=(z_{\tau})_{\tau \in \Sigma} \in  \prod_{\tau \in \arch} \HH_{\siga_\taut \times \sigb_\taut}.
\end{equation*}
By \cite[12.1]{Shimura}, $\cH$ is the irreducible (Hermitian) symmetric domain for $SU(\bR)$. By the classification of Hermitian symmetric domains and \cite[Corollary~5.8]{Milne2}, $\HH$ is uniquely determined by the adjoint group of a connected component of $GU(\bR)$.  Recall from Section \ref{moduli-section} that we can identify $\fh$ with a finite union of copies of the symmetric domains for the identity component of $GU(\IR)$. Hence $\fh$ can be identified with a finite (disjoint) union of copies of $\cH$.

In order to define the desired transformation properties that automorphic forms should satisfy, we need to introduce a few more definitions. 
Using the above notation, for $g \in GU_+(\bR)$ and $z=(z_\tau)_{\tau \in\arch} \in \HH$ the \textit{factors of automorphy} 
for each $\{\tilde{\tau},\tilde{\tau}^c\} \subset \Sigma_K$ above $\tau \in \Sigma$ are defined by
\begin{equation*}
 \mu_{\tilde{\tau}}(g,z):=c_{g,\tau}z_\tau + d_{g,\tau} \mbox{ and } \mu_{\tilde{\tau}^c}(g,z) := \ov b_{g,\tau} {^t z_\tau} + \ov a_{g,\tau},
\end{equation*}
and the \textit{scalar factors of automorphy} are
\begin{equation*}
 j_\tau(g,z):=\det(\mu_\tau(g,z)) \mbox{ for } \tau \in \archK .
\end{equation*}

So far, we have considered the case in which $a_{+\tilde\tau}a_{-\tilde\tau}\neq 0$.  Now, suppose $a_{+\tilde\tau}a_{-\tilde\tau}=0$.  In this case, 
$\HH_{\siga_\taut \times \sigb_\taut}$ is defined to be the element $0$, with the group acting trivially on it.  Following \cite[Section 3.3]{Shimura}, if $a_{-\tilde\tau}=0$, we define 
\begin{align*}
\mu_{\tilde\tau^c}(g, z)&=\overline{g}\\
\mu_{\tilde\tau}(g, z)&=1\\
j_\tau(g, z) &=1,
\end{align*}
and if $a_{+\tilde\tau}=0$, we define
\begin{align*}
\mu_{\tilde\tau}(g, z) &= g\\
\mu_{\tilde\tau^c}(g, z) & = 1\\
j_\tau(g, z) &= \det(g).
\end{align*}

\begin{remark}
By \cite[Equation (1.19)]{shunitary}, for all $\tau \in \Sigma$ and $g \in GU_+(\bR)$,
\begin{align*}
\det(\mu_{\tilde{\tau}^c}(g,z)) = \det\left(g \right)^{-1}\nu\left(g\right)^{a_{+,\tilde{\tau}}}\det(\mu_{\tilde{\tau}}(g,z)).
\end{align*} 
\end{remark}
Define 
\begin{equation*}
 M_g(z):=(\mu_{\tilde{\tau}^c}(g,z),\mu_{\tilde{\tau}}(g,z))_{\tau \in \arch} \quad  \in  \prod_{\tau \in \arch} \Mat_{\siga_{\tilde{\tau}}\times\siga_{\tilde{\tau}}} \times \Mat_{\sigb_{\tilde{\tau}}\times\sigb_{\tilde{\tau}}} 
\end{equation*}
If $\rep: \Levin(\bC) = \prod\limits_{\tau \in \arch} \GL_{\siga_{\tilde\tau}}(\bC) \times \GL_{\sigb_{\tilde\tau}}(\bC) \ra \GL(X)$ is a rational representation into a finite-dimensional complex vector space $X$, $f: \HH \ra X$ a map and $g \in GU_+(\bR)$, then denote by $f||_\rep g: \HH \ra X$ and $f|_\rep g: \HH \ra X$ the maps given by
\begin{align*}
 (f||_\rep g)(z)&:=\rep(M_g(z))^{-1}f(gz)\\
f|_\rep g&:= f||_\rep \left(\nu(g)^{-1/2}g\right)
\end{align*}
for all $z\in \HH$.  Note that for all $g\in U(\bR)$,
\begin{align*}
f|_\rep g &= f||_\rep g.
\end{align*}

Associated to an $\cO_F$-lattice $L_F$ in  $L\otimes_\bZ \bQ$ 
and an integral $\cO_F$-ideal $\fc$, we define the subgroup
\begin{equation*}
\Gamma(L_F,\fc):=\{ g \in GU_+(\bQ) \, | \, {^tL_F}g={^tL_F} \mbox{ and } {^tL_F}(1-g) \subset \fc {^tL_F} \} .
\end{equation*}
Then a \textit{congruence subgroup} $\Gamma$ of $GU_+(\bQ)$ is a subgroup of $GU_+(\bQ)$ that contains $\Gamma(L_F, \fc)$ as subgroup of finite index for some choice of $(L_F, \fc)$ as above.

\begin{definition} \label{def-first-autom-forms}
 Let $\Gamma$ be a congruence subgroup of $GU_+(\bQ)$, $X$ a finite-dimensional complex vector space and $\rep: \Levin(\bC) \ra \GL(X)$ a rational representation. A function $f: \HH \ra X$ is called a \textit{(holomorphic) automorphic form }of weight $\rep$ with respect to $\Gamma$ if it satisfies the following properties
 \begin{enumerate}
  \item $f$ is holomorphic,
 \item $f||_\rep \gamma = f$ for every $\gamma$ in $\Gamma$,
 \item if $\Sigma$ consists of only one place $\tau$ and $(a_{+\tilde\tau},a_{-\tilde\tau})=(1,1)$, then $f$ is holomorphic at every cusp.
 \end{enumerate}
We will call a function $f:\HH \ra X$ that satisfies property (2), but not necessarily (1) and (3) an \textit{automorphic function}. \end{definition}

\begin{remark} \label{remark-Koecher} Note that if we are not in the case in which both $\Sigma$ consists of only one place $\tau$ and $(a_{+\tilde\tau},a_{-\tilde\tau})=(1,1)$, then Koecher's principle implies that an automorphic form is automatically holomorphic at the boundary.  (See \cite[Thm.~2.5]{lankoecher} for a very general version.)
\end{remark}

\begin{remark}
Sometimes, in the definition of an automorphic form, the second condition of Definition \ref{def-first-autom-forms} is replaced by $f|_\rep \gamma = f$ for every $\gamma$ in $\Gamma$.  The condition that arises from geometry, though, is $f||_\rep \gamma = f$.  Since our main results and proofs are geometric (and since we want this definition of automorphic forms to agree with the geometric definitions we give later), we require $f||_\rep \gamma = f$ instead of $f|_\rep \gamma = f$ in this paper.
\end{remark}

\subsubsection{Weights of an automorphic form} \label{Section-weights}

The irreducible algebraic representations of $\Levin =\prod_{\tau \in \arch} \GL_{\siga_{\tilde\tau}} \times \GL_{\sigb_{\tilde\tau}}$ over $\bC$ are in one to one correspondence with dominant weights of a maximal torus $T$ (over $\bC$). 
More precisely, let $T$ be the product of the diagonal tori $T_{\siga_\taut} \times T_{\sigb_{\taut}}$ for $\tau \in \arch$. For $1 \leq i \leq \siga_\taut+\sigb_{\taut}$, let $\eps^{\tau}_i$ in $X(T):=\Hom_\bC(T,\bG_m)$ be the character defined by
	$$T(\bC) = \prod_{\sigma \in \arch} T_{\siga_\sigmat}(\bC) \times T_{\sigb_{\sigmat}}(\bC) \ni \diag(x^{\sigma}_1, \hdots, x^{\sigma}_{\siga_\sigmat+\sigb_\sigmat})_{\sigma \in \arch} \mapsto x^\tau_i \in \bG_m(\bC).$$ 
These characters form a basis of the free $\bZ$-module $X(T)$, and we choose $\Delta=\{ \alpha^{\tau}_i := \eps^{\tau}_i - \eps^{\tau}_{i+1} \}_{\tau \in \arch, 1 \leq i < \siga_\taut+\sigb_\taut, i \neq \siga_\taut}$ as a basis for the root system of $\Levin$. 
Then the {\em dominant weights} of $T$ with respect to $\Delta$ are $X(T)_+=\{ \kappa \in X(T) \, | \, \< \kappa, \check{\alpha} \> \geq 0 \, \forall \alpha \in \Delta  \}$, and using the above basis of $X(T)$ they can be identified as follows:
	$$X(T)_+ \cong \{(n^{\tau}_1,\hdots, n^{\tau}_{\siga_{\taut}+\sigb_{\taut}})_{\tau \in \arch} \in  \prod\limits_{\tau \in \arch} \bZ^{\siga_{\taut} + \sigb_\taut} : n^{\tau}_i \geq n^{\tau}_{i+1} \forall i \neq \siga_\taut\}.$$
For such a dominant weight $\kappa$, let $\rho_{\kappa}: \Levin_{\bC} \rightarrow M_{\rho}$ denote the irreducible algebraic representation of highest weight $\kappa$.  (See, for example, \cite[Part~II.  Chapter~2]{Jantzen}.) 
We call $f$ an automorphic form of weight $\kappa$, where $\kappa \in X(T)_+$, if $f$ is an automorphic form of weight $\rho_{\kappa}$.

\begin{remark}
	In view of later generalizations, note that $\rho_\kappa$ can be defined as an algebraic (i.e. schematic) representation of the algebraic group $\Levin$ over the integers $\bZ$. More precisely, let $T$ be the maximal split torus of $\Levin$ extending the diagonal torus over $\bC$ constructed above and note that $\Hom_\bZ(T,\bG_m)=\Hom_\bC(T,\bG_m)$, i.e. we can view $\kappa \in X(T)_+$ as a character of $T$ defined over $\bZ$. Let $B$ be a Borel subgroup containing $T$ (corresponding to upper triangular matrices in $\Levin$) and $B^-$ the opposite Borel, and denote by $N$ and $N^-$ the unipotent radicals of $B$ and $B^-$, respectively.  Then $\kappa$ can be viewed as a character of $B^-$ acting trivially on $N^-$ via the quotient $B^- \twoheadrightarrow T$, and the universal representation $\rho_\kappa$ of highest weight $\kappa$ is given by 
	$$\Ind_{B^-}^\Levin (-\kappa) = \{ f: \Levin/N^- \rightarrow \bA^1 \, | \, f(ht) = \kappa(t)^{-1}f(h), t \in T\},$$
	 where $\bA^1$ is the affine line and on which $\Levin$ acts via $$\Levin \ni h: f(x) \mapsto \rho_{\kappa}(h)f(x) = f(h^{-1}x),$$ see \cite{Jantzen}. This representation is irreducible of highest weight $\kappa$ over any field of characteristic zero.
	\end{remark}

\subsubsection{Unitary domains and moduli problems}

We will now reformulate Definition \ref{def-first-autom-forms} in order to generalize it in Section \ref{classical-aut-forms}. We quickly recall the correspondence between quadruples described in Section \ref{moduli-section} and points of $\cH$.

Fix a PEL datum $(K,c,L,\langle,\rangle,h)$ and a neat open compact subgroup $\cpct \subset GU(\bA^\infty)$; let $\modulispace_{\cpct}(\bC)$ denote the complex Shimura variety of level $\cpct$ associated to the PEL datum discussed in Section \ref{moduli-section}, and let $S_{\cpct}(\bC)$ be a connected component in $M^{(1)}_{\cpct}(\bC)$. Recall that $V = L \otimes_{\bZ} \bR$.
We now describe the identification between elements $z \in \Gamma\backslash \cH$ and abelian varieties $\uA \in S_{\cU}^{}(\bC)$ given by \cite[Chapter I]{Shimura}. Shimura defines for $z \in \cH$ an $\bR$-linear isomorphism $p_z: V \rightarrow \bC^g$ which induces a Riemann form on $\bC^g$:
	$$E_z(p_z(x),p_z(y)) = \langle x,y\rangle \quad x,y \in V.$$
This implies that $A_z = \bC^g/p_z(L)$ is an abelian variety together with a polarization $\polarization_z$ corresponding to $E_z$. For $k \in K$, define $\ENDO_z(k)$ to be the element of $\End_{\bQ}(A_z)$ induced by the action of $h(k_{{\tau}})$ {acting on $V_{\tau}$} and let $\level_z$ denote the $\cU$-orbit of isomorphisms $V \otimes \bA^\infty \cong H_1(A_z,\bA^\infty)$ induced by $p_z$. Altogether, $p_z$ gives rise to the Riemann form $E_z$, the following commutative diagram:
\begin{equation}  \label{structure-diagram}
 \xymatrix{ 0 \ar[r] & L \ar[r] \ar[d] & V \ar[r]\ar[d] & V/L \ar[r]\ar[d] & 0 \\
             0 \ar[r] & \Lambda \ar[r] &  \bC^g \ar[r] & A  \ar[r] & 0,
  }
\end{equation}	
and $\uA_z := (A_z,\ENDO_z,\polarization_z,\level_z)$. Shimura \cite[Theorem 4.8]{Shimura} (together with \cite[Proposition 4.1]{Milne}) further proves that 
\begin{proposition} \label{thm-shimura} 
	For each $z \in \cH$, $\ul A_z \in S_{\cpct}^{}(\bC)$ for some neat open $\cpct$. Conversely, if $\uA \in S_{\cpct}(\bC)$ for some $\cpct$, then there is a $z \in \cH$ such that $\uA$ is equivalent to $\ul A_z$. Furthermore $\ul A_z$ and $\ul A_w$ for $z,w \in \HH$ are equivalent if and only if $w=\gamma z$ for some $\gamma \in \Gamma_\cpct := \cpct\cap GU_+(\bQ).$
\end{proposition}

\subsubsection{Second definition of classical automorphic functions}
For $\uA \in S_{\cU}^{}(\bC)(\bC)$, let $\uo = H^1(A,\bZ) \otimes \bC$ which has a decomposition into $\uo^{\pm}$ depending on the (induced) action of $h(i)$ on $\uo$. Set
\begin{equation*}
	{\EA^\pm}_\tau=\Isom_\bC(\bC^{\sigpm_\taut},\uo_\tau^\pm) \quad \mbox{ and } \quad \EA=\bigoplus_{\tau \in \arch}({\EA^+}_\tau \oplus {\EA^-}_\tau).
\end{equation*}
The group $\GL_{\sigpm_\taut}(\bC)$ acts on ${\EA^\pm}_\tau$ by 
\begin{equation*}
	(g.f)(x)=f(\transp{g}x) \, \, \mbox{ for } g \in \GL_{\sigpm_\taut}(\bC), f \in {\EA^\pm}_\tau, x \in \bC^{\sigpm_\taut},
\end{equation*}
which induces a (diagonal) action of $\Levin(\bC)$ on $\EA$.

For $z \in \cH$, the map $p_z$ induces a choice of basis on $H_1(A_z,\bZ)$ via the identification with $p_z(L)$, which by duality induces a choice of basis $\cB_z$ of $\uo$. This is equivalent to giving an element $\l_z$ of $\EAz{z}$.

\begin{lemma} \label{lemma-second-def-easy} Let $\rep: \Levin(\bC) \ra \GL(X)$ be a rational representation. Then there exists a one-to-one correspondence between automorphic functions of weight $\rep$ with respect to $\Gamma_\cpct$ and the set of functions $F$ from pairs $(\ul A, l)$, where $\ul A \in S_{\cU}^{}(\bC)$ and $l \in \EA$, to $X$ satisfying
	\begin{equation} \label{eqn-F-transformation}
		F(\ul A, gl)=\rep(\transp{g})^{-1}F(\ul A, l) \mbox{ for all } g \in \Levin(\bC).
	\end{equation}
	The bijection is given by sending a function $F$ to the automorphic function $f_F: z \mapsto F(\ul A_z, l_z)$.
\end{lemma}

\begin{proof} (sketch) 
	We will first show that the map $F \mapsto (f_F: z \mapsto F(\ul A_z, l_z))$ is well defined and afterwards construct an inverse for it.
	To do the former, let $z \in \HH$, $\gamma \in \Gamma_{\cpct}$. We need to show that $F(\ul A_z, l_z)=\rep(M_\gamma(z))^{-1}F(\ul A_{\gamma z}, l_{\gamma z})$. It follows from \cite[page 27]{Shimura} that $\transp{M_\gamma(z)}$ maps $p_{\gamma z}(L)$ to $p_{z}(L)$ and defines an isomorphism between $\ul A_{\gamma z}$ and $\ul A_{z}$. Note that under this isomorphism, $l_{\gamma z}$ of $\uo$ gets mapped to $\transp{M_\gamma(z)}^{-1} l_z$. 
	Hence using property (\ref{eqn-F-transformation}), we obtain
	\begin{equation*}
		F(\ul A_{\gamma z},l_{\gamma z})=F(\ul A_z, \transp{M_\gamma(z)}^{-1}l_z)=\rep(M_\gamma(z))F(\ul A_z, l_z)
	\end{equation*}
	as desired.
	
	We define the inverse of $F \mapsto f_F$ as follows: Let $f$ be an automorphic function of weight $\rep$ with respect to $\Gamma_{\cpct}$, and $(\ul A, l)$ where $\uA \in S_{\cU}^{}(\bC)$ and $l \in \EA$. Then, by Proposition \ref{thm-shimura}, there exists $z \in \HH$ such that $\ul A_z$ is isomorphic to $\ul A$, and there exists a unique $g \in \Levin(\bC)$ such that $l=gl_z$. We define $F_f(\ul A, l)=\rep(\transp{g})^{-1}f(z)$. By the transformation property of automorphic function 
	we obtain analogously to above that $f \mapsto F_f$ is well-defined and $F_f$ satisfies (\ref{eqn-F-transformation}). Moreover, $f \mapsto F_f$ is obviously an inverse of $F \mapsto f_F$.
\end{proof}

Automorphic functions of weight $\rep$ with respect to congruence subgroups $\Gamma$ that strictly contains $\Gamma_{\cpct}$ can be characterized in the same style as follows:

\begin{lemma} \label{lemma-second-def} Let $\rep: \Levin(\bC) \ra \GL(X)$ be a rational representation, 
  and $\Gamma$ a congruence subgroup of $GU_+(\bQ)$ containing $\Gamma_{\cpct} = \cpct \cap GU_+(\bQ)$.
	Then there exists a one to one correspondence between automorphic functions of weight $\rep$ with respect to $\Gamma$ and the set of functions $F$ from pairs $(\ul A, l)$, where $\ul A = \PELtuple \in S_{\cpct}^{}(\bC)$  and $l \in \EA$, to $X$ satisfying
	\begin{equation*}
		F(\ul A, gl)=\rep(\transp{g})^{-1}F(\ul A, l) \mbox{ for all } g \in \Levin(\bC),
	\end{equation*}
	and such that for all $\gamma \in \Gamma$ and $z \in \HH$, we have
	\begin{equation} \label{gamma-invariance}
		F(\ul A_z, l_z)=\rep(M_\gamma(z))^{-1}F(\ul A_{\gamma z}, l_{\gamma z}). 
	\end{equation}
	It suffices to check condition (\ref{gamma-invariance}) for a set of representatives of $\Gamma/\Gamma_\cpct$.
\end{lemma}
This lemma follows easily from Lemma \ref{lemma-second-def-easy} and Definition \ref{def-first-autom-forms}.

\subsubsection{Algebraic definition of classical automorphic functions} 
Finally, we would like to view automorphic forms as functions on abelian varieties on $S_{\cpct}(\bC)$. For this, we define the \textit{contracted product} $\EAo =\EA\times^\rep X$
of $\EA$ and $X$ to be the product $\EA \times X$ modulo the equivalence relation given by $\left(l, v \right)\sim\left(gl,\rep(\transp{g})^{-1}v \right)$ for $g \in \Levin(\bC)$. Note that for $g \in GU_+(\bQ)$, the identification $H_1(A_z,\bZ)\otimes \bC =\bC^g = H_1(A_{g z},\bZ)\otimes \bC$
 induces an identification $\iota_g: \EAz{z} \ra \EAz{g z}$, and we can define the isomorphism $i_{g}:\EAoz{z} \ra \EAoz{g z}$ by $(l, v) \ra (\iota_g(l),\rep(M_g(z))v)$. Note that $i_g$ is the identity for $g \in \Gamma_{\cU}$. 
It is an easy exercise to see that Lemma \ref{lemma-second-def} can be reformulated as follows. 
\begin{lemma} \label{lemma-equivalent-def3}
	Let $\rep: \Levin(\bC) \ra \GL(X)$ be a rational representation, 
	 and $\Gamma$ a congruence subgroup of $GU_+(\bQ)$ containing $\Gamma_{\cpct}$.
	Then there exists a one-to-one correspondence between automorphic functions of weight $\rep$ and level $\Gamma$ and the set of functions $\wt F$ from $\ul A \in S_{\cpct}(\bC)$ to $\EAo$
	satisfying
	\begin{equation} \label{gamma-invariance-three}
		i_\gamma \left( \wt F(\ul A_z) \right)=\wt F(\ul A_{\gamma z}) \mbox{ for all } z \in \HH, \gamma \in \Gamma.
	\end{equation}
\end{lemma}
\begin{remark} \label{rmk-equivalent-def3}
	Note that by Proposition \ref{thm-shimura} giving a function $\wt F$ from $S_{\cpct}^{}(\bC)$ to $\EAo$ satisfying (\ref{gamma-invariance-three}) is the same as giving a global (point-theoretic) section of the vector bundle $\EAoz{z}$ over $\Gamma\backslash \HH$.
\end{remark}

\subsection{(Classical) algebraic automorphic forms}\label{classical-aut-forms}

In Section \ref{autformsC-section}, we considered automorphic forms over $\bC$.  Building on the discussion from Section \ref{autformsC-section}, we now consider automorphic forms over other base rings.  The approach in this section is similar to the approach in \cite[Section 2.5]{EDiffOps} and \cite[Section 1.2]{kaCM}.  Note that \cite{EDiffOps} only considers the case in which the signature  is $(\siga_{\tau},\sigb_{\tau})_{\tau \in \archK}$ with $\siga_{\tau}=\sigb_{\tau}$ for every $\tau \in \archK$, but the definitions from \cite[Section 2.5]{EDiffOps} carry over to the general case with only trivial modifications.
In order not to worry about the holomorphy conditions at cusps, we exclude the case of $\Sigma =\{ \tau \}$ with $(a_{+\taut},a_{-\taut})=(1,1)$ from the discussion in this section, compare Remark \ref{remark-Koecher}. 

For any neat open compact subgroup $\cpct$, consider the integral model $\cM_{\cpct}/\Oep$ introduced in Section \ref{PELmoduli-section}. For any scheme $S$ over $\Spec(\Oep)$, we put
\begin{align*}
	\cM_{\cpct,S}:=\cM_{\cpct}\times_{\Oep}S.
\end{align*} 
When $S = \Spec(R)$ for a ring $R$, we will often write $\cM_{\cpct,R}$ instead of $\cM_{\cpct,\Spec(R)}$. {If $\bW$ denotes the ring of Witt vectors associated to $\overline{\bF}_p$, then consider $\cM_{\cU,\bW}$ (note that since $p$ splits completely, we can base change to $\bW$). In the sequel, we consider (locally noetherian) schemes $S$ over $\bW$. {Note that instead of working over $\Witt$, we could also work over $\cO_{E',(p)}$ in this section, where $E'$ is a finite extension of $E$ that contains $\tau(K)$ for all $\tau \in \Sigma_K$.}}

For any $S$-point $\uA = (A,\ENDO,\polarization,\level^{p})$ of $\cM_{\cpct,\bW}$, let $\uo_{\uA/S}$ denote the locally free ${\bW \otimes \cO_K}$-module defined as the pullback via the identity section of the relative differentials. We have a {natural} decomposition $\uo_{\uA/S} = \bigoplus_{\tau \in \arch} (\uo_{\uA/S,\taut}^+ \bigoplus \uo_{\uA/S,\taut}^-)$ where $\uo_{\uA/S,\taut}^{\pm}$ is rank $\siga_\taut$ and $\sigb_\taut$, respectively. {Note that the element $x \in \cO_K$ acts on $\uo_{\uA/S,\taut}^+$ (resp. $\uo_{\uA/S,\taut}^-$) via {$\taut(x)$} (resp. {$\taut^c(x)$}). {(Here, we view {$\taut$}  as an embedding of $K$ which factors through ${\rm Frac}(\bW)$)}.}
 Define 
$$\EAR^\pm :=\bigoplus_{\tau\in \arch}\Isom_{\mathcal{O}_S}\left(\mathcal{O}_S^{\sigpm_\taut}, \uo_{\uA/S, \tau}^\pm\right) \qquad 
\mbox{and} \qquad
\EAR :=\EAR^+\oplus\EAR^-,
$$
respectively.
Let $R$ be {a $\Witt$}-algebra, and consider an algebraic representation $\rho$ of $\Levin_R$ into a finite free $R$-module $M_{\rho}$.

\begin{definition}[First Equivalent Definition of Algebraic Automorphic Forms]\label{algauto-defi1}
	An automorphic form of weight $\rho$ and level $\cpct$ defined over $R$ is a function $f$
	\begin{align*}
		\left(\uA, \lambdaiso\right)\mapsto f\left(\uA, \lambdaiso\right)\in (M_\rho)_{R'}
	\end{align*}
	defined for all $R$-algebras $R'$, $\uA$ $\in \cM_{\cpct}(R')$, and  $\lambdaiso\in\EARp$, such that all of the following hold:
	\begin{enumerate}
\item{$f\left(\uA, \alpha\lambdaiso\right) = \rho\left(\left({ }^t\alpha\right)^{-1}\right)f\left(\uA, \lambdaiso\right)$ for all $\alpha\in\Levin\left(R'\right)$ and all $\lambdaiso\in\EARp$}
		\item{The formation of $f\left(\uA, \lambdaiso\right)$ commutes with extension of scalars $R_2\rightarrow R_1$ for any $R$-algebras $R_1$ and $R_2$.  More precisely, if $R_2\rightarrow R_1$ is a ring homomorphism of $R$-algebras, then
			\begin{align*}
				f\left(\uA\times_{R_1}{R_2}, \lambdaiso\otimes_{R_1} 1\right) = f\left(\uA, \lambdaiso\right)\otimes_{R_1} 1_{R_2}\in (M_\rho)_{R_2}
			\end{align*}
		}
	\end{enumerate}
\end{definition}
In order to give a different equivalent definition (Definition \ref{algauto-defi2} below), we define for any algebraic representation $\rho$ of $\Levin_R$ into a finite free $R$-module $M_{\rho}$ and $R$-algebra $R'$
\begin{align*}
	\cE_{(\uA/R',\rho)}=\cE_{\uA/R'}\times^{\rho}(M_\rho)_{R'}:=\left(\cE_{\uA/R'}\times(M_\rho)_{R'}\right)/\left(\lambdaiso, m\right)\sim\left(g\lambdaiso, \rho({ }^tg^{-1})m\right),\end{align*}
where $g \in \Levin(R')$ acts on $\cE_{\uA/R'}$ by precomposing with ${ }^tg$.

\begin{definition}[Second Equivalent Definition of Algebraic Automorphic Forms]\label{algauto-defi2}
	An automorphic form of weight $\rho$ and level $\cpct$ defined over a {$\Witt$}-algebra $R$ is a function $\tilde{f}$ 
	\begin{align*}
		\uA\mapsto \tilde{f}(\uA)\in\cE_{(\uA/R',\rho)}
	\end{align*}
	defined for all $R$-algebras $R'$ and $\uA\in\cM_{\cpct}(R')$ such that the formation of $\tilde{f}(\uA)$ commutes with extension of scalars $R_2\rightarrow R_1$ for any $R$-algebras $R_1$ and $R_2$.  More precisely, if $R_2\rightarrow R_1$ is a ring homomorphism of $R$-algebras, then
			\begin{align*}
				\tilde{f}\left(\uA\times_{R_1} R_2\right) = \tilde{f}\left(\uA\right)\otimes_{R_1}1_{R_2}.
			\end{align*}
\end{definition}

\begin{remark}
	The equivalence between Definition \ref{algauto-defi1} and Definition \ref{algauto-defi2} 
	is given by
	\begin{align*}
		\tilde{f}(\uA) = \left(\lambdaiso, f\left(\uA, \lambdaiso\right)\right)
	\end{align*}
	for all abelian varieties $\uA/R$ (corresponding to $\cM(R)$) and $\ell\in\cE_{\uA/R}$.
\end{remark}

Finally, we want to view automorphic forms as global sections of a certain sheaf. In order to do so, let $\Auniv =  \left(A,\ENDO,\polarization, \level^{p}\right)^{\univ}$ denote the universal abelian variety over $\cM_{\cpct,\bW}$, 
 and define the sheaf
  $$\cE=\cE_\cpct:= \bigoplus_{\tau\in \arch}\Isoms_{\mathcal{O}_{\cM_{\cpct,\bW}}}\left(\mathcal{O}_{\cM_{\cpct,\bW}}^{\sigp_\taut}, \uo_{\Auniv/{\cM_{\cpct}}, \tau}^+\right) 
  \oplus \bigoplus_{\tau\in \arch}\Isoms_{\mathcal{O}_{\cM_{\cpct,\bW}}}\left(\mathcal{O}_{\cM_{\cpct,\bW}}^{\sigm_\taut}, \uo_{\Auniv/{\cM_{\cpct,\bW}}, \tau}^-\right),$$
   i.e. for every open immersion $S \hookrightarrow \cM_{\cpct,\bW}$, we set $\cE_\cpct(S)=\cE_{\Auniv_S/S}$.
 Moreover, for any algebraic representation $\rho$ of $\Levin_R$ over a free finite $R$-module $M_\rho$, we define the sheaf $\cE_{\rho}=\cE_{\cpct,\rho} := \cE \times^{\rho} M_{\rho}$, i.e. for each open immersion $\Spec R' \hookrightarrow \cM_{U,\bW}$, set $\cE_{\cpct,\rho}(R')=\cE_{(\Auniv_{R'}/R',\rho)}$.

\begin{definition}[Third Equivalent Definition of Algebraic Automorphic Forms]\label{algauto-defi3}
	An automorphic form of weight $\rho$ and level $\cpct$ defined over $R$ is a global section of the sheaf $\cE_{\cpct,\rho}$ on $\cM_{\cpct,R}$.
\end{definition}

\begin{remark}When we are working with a representation $\rho$ which are uniquely determined by its highest weight $\kappa$, we shall sometimes write $\cE_{\cpct,\kappa}$ or $\cE_\kappa$, in place of $\cE_{\cpct,\rho}$. 
\end{remark}
\begin{remark}
		Usually automorphic forms are defined over a compactification of $\cM_{\cpct,R}$, but in our case, i.e. excluding the case of $\Sigma$ consisting only of one place $\tau$ and $(a_{+\tau},a_{-\tau})=(1,1)$, both definitions are equivalent by Koecher's principle.
	\end{remark}

\subsubsection{Comparison with classical definition of complex automorphic forms}
Having defined algebraic automorphic forms over general base rings, we will show that in the special case of the base ring being $\bC$ the definition coincides with the classical definition of complex automorphic forms given in Section \ref{autformsC-section}. 

For an integer $N$, we define $\cpct_N$ to be a compact open subgroup of $GU(\bA^\infty)$ such that 
	$$GU_+(\bQ) \cap \cpct_N = \Gamma(N) := \{(g,\nu) \in GU_+(\bQ) : g \equiv 1 \mod N\}.$$
	
\begin{proposition} \label{prop-comparison} Let $N$ be a large enough integer so that $\cpct_N$ is neat, and let $\rho$ be an algebraic representation of $\Levin$ over $\bC$. Then there is a bijection between the (algebraic) automorphic forms of weight $\rho$ defined in Definition \ref{algauto-defi3} as global sections of $\cE_\rho$ on ${{\modulispace}_{\cpct_N}}(\bC)$ and a finite set of holomorphic automorphic forms of weight $\rho$ with respect to $\Gamma(N)$ as defined in Definition \ref{def-first-autom-forms} in Section \ref{autformsC-section}.
\end{proposition}

\begin{proof} (sketch)
From the the classification of Hermitian symmetric spaces mentioned in Section \ref{autformsC-section}, one can deduce that $\modulispace^{(i)}_{\cpct_N}(\bC)$ (as defined in \S\ref{moduli-section}) is isomorphic to a finite union of copies of $\Gamma \backslash \HH$.  By GAGA
  and Lemma \ref{lemma-equivalent-def3} together with Remark \ref{rmk-equivalent-def3}, we conclude that the global sections of $\cE_\rho=\cE_{\cpct, \rho}$ on ${\modulispace_{\cpct_N}}({\bC})$ are in one-to-one correspondence with a finite set of holomorphic automorphic forms of weight $\rho$ with respect to $\Gamma(N)$, one for each connected component of $M_{\cpct_N}^{(i)}(\bC)$ for all $i$.
\end{proof}

\section{$p$-adic theory}\label{padic-section} 
Section \ref{Igusalevel} introduces the Igusa tower, a tower of finite \'etale Galois coverings of the ordinary locus of $\cM_\cpct$, which we denote by $\cM$ from now on, since we fix the neat level $\cpct$ throughout the rest of the paper.  Section \ref{padicaut-section} introduces $p$-adic automorphic forms, which arise as global sections of the structure sheaf of the Igusa tower.

We are mainly following \cite[Section~8]{hida} and \cite{hidairreducibility}.
\subsection{The Igusa tower over the ordinary locus}\label{Igusalevel} 
Recall that our Shimura varieties with hyperspecial level at $p$ admit integral models $\cM$ over $\cO_{E,(p)}$. These Shimura varieties have a neat level away from $p$, but we suppress it from the notation since the tame level won't affect the geometry of our integral models. As we will see below, the geometry of the integral models $\cM$ is governed by the $p$-divisible group of the universal abelian variety over $\cM$.

 In order to guarantee that the ordinary locus (defined below) over the special fiber of $\cM$ is nonempty, we make the following assumption: the prime $p$ splits completely in the reflex field $E$ (in ~\cite{wedhorn} Wedhorn proves that such an assumption is both necessary and sufficient). In this case, the ordinary locus is open and dense in the special fiber of $\cM$. 
Choose a place $P$ of $E$ above $p$ and let $E_P$ be the corresponding completion of $E$, with ring of integers $\cO_{E_P}$ and residue field $k$.
By abuse of notation, we will still denote the base change of $\cM$ to $\cO_{E_P}$ by $\cM$.  Let $S$ be a scheme of characteristic $p$.

\begin{defn}An abelian variety $A/S$ of dimension $g$ is ordinary if for all geometric points $s$ of $S$, the set $A[p](s)$ has $p^g$ elements.
\end{defn}

For every abelian variety $A/S$, the Hasse invariant $\mathrm{Ha}_{p-1}(A/S)$ is a global section of $\omega_{A/S}^{\otimes (p-1)}$, where $\omega_{A/S}$ is the top exterior power of the pushforward to $S$ of the sheaf of invariant differentials on $A$. It is easy to show that an abelian variety $A$ is ordinary if and only if $\mathrm{Ha}(A/S)$ is invertible. (We now sketch an argument for this, as in \cite[Lemma III.2.5]{scholze}: the Hasse invariant, which corresponds to pullback along the Verschiebung isogeny, is invertible if and only if Verschiebung is an 
isomorphism on tangent spaces, which happens if and only if Verschiebung is finite etale.  A degree computation shows that this is equivalent to the condition that $A $ be ordinary.) We define the ordinary locus \[\Mbarord\subset \overline{\cM}:= \cM\times_{\cO_{E_P}} k\] to be the complement of the zero set of the Hasse invariant. 
Since $p$ splits completely in $E$, the nonemptiness of $\Mbarord$ follows from~\cite{wedhorn}. In fact, Wedhorn proves something stronger, namely that $\Mbarord$ is dense in the special fiber $\overline{\cM}$. We also define the ordinary locus $\Mord$ over $\cO_{E_P}$ to be the complement of the zero set of a lift of some power of the Hasse invariant.

In addition, we define $\Sord$ over $\Witt$ to be a connected component of $\Mord_\Witt := \Mord\times_{\cO_{E_P}}\Witt$. (Recall that $p$ splits completely in $E$, so the base change to $\Witt$ makes sense.) In other words, $\Sord$ is the ordinary locus of one connected component $\mathcal{S}$ of $\mathcal{M}_\Witt$.

\begin{remark} We note that we could define the ordinary locus in an alternate way, using the stratification of $\overline{\cM}$ in terms of the isogeny class of the $p$-divisible group $\cG:=\cA_{\overline{\cM}}[p^\infty]$, where $\cA_{\overline{\cM}}$ denotes the universal abelian variety over $\overline{\cM}$. The isogeny class of the $p$-divisible group $\cG$ (equipped with all its extra structures) defines a stratification of $\overline{\cM}$ with locally closed strata, which is called the \emph{Newton stratification}. The ordinary locus, corresponding to the constant 
 isogeny class $(\mu_{p^\infty})^g\times (\mathbb{Q}_p/\mathbb{Z}_p)^g$, is the unique open stratum.
\end{remark}

Let $\cA:=\cA_{\Mord}$ be the universal ordinary abelian variety over $\Mord$. Pick a $\Witt$-point $x$ of $\Sord$, with an $\overline{\mathbb{F}}_p$-point $\bar x \in \Sord$ below it.
We can identify $L_p$ (defined as in Section \ref{unitarygroups-section}) with the $p$-adic Tate module of the $p$-divisible group $\cG_x$, i.e. the $p$-adic Tate module of $\cA_x$. Choose such an identification $L_p\simeq T_p\cA_x[p^\infty]$, compatible with the $\cO_K$-action and with the Hermitian pairings. The kernel of the reduction map \[T_p\cA_x[p^\infty]\to T_p\cA_{\bar{x}}[p^\infty]^{\mathrm{\et}}\] corresponds to an $\cO_K$-direct summand of $\cL \subset L_p$.  Note that the lattice $\cL$ is independent of the choice of $x$ inside the connected component $\Sord$, and hence the different connected components of $\Mord$ can be labeled by lattices $\cL$.
Moreover, using the self-duality of $L_p$ under the Hermitian pairing $\langle\cdot,\cdot\rangle$ and the compatibility with the Weil-pairing, we can identify the dual $\cL^\vee$ of $\cL$ with the orthogonal complement of $\cL$ inside $L_p$.

\begin{remark}\label{explicitdecomposition} By considering the primes in $\cmfield^+$ above $p$ individually, we can write down an explicit formula for $\cL$, 
	using the fact that each such prime splits from $\cmfield^+$ to $K$. The exact formula for $\cL$ as an $\cO_K$-module will depend on the set of signatures of the unitary similitude group $GU(\mathbb{R})$. More precisely, recall that for each embedding $\tau : \cmfield\hookrightarrow \mathbb{C}$, $\left(\siga_\tau,\sigb_\tau\right)$ is the signature of $GU$ at the infinite place $\tau$. Choose an isomorphism $\iota_p:\mathbb{C}\toisom \mathbb{\bar Q}_p$. By composing with $\iota_p$, each $\tau$ determines a place of $\cmfield$ above $p$. Let $p=\prod_{i=1}^r \mathfrak{p}_i$ be the decomposition of $p$ into prime ideals of $\cmfield^+$. Each $\mathfrak{p}_i$ splits in $K$ as $\mathfrak{p}_i=\mathfrak{P}_i\mathfrak{P}^c_i$, where $\mathfrak{P}_i$ lies above the prime $w$ of $F$. The $i$-term of $L_p$ (obtained from the decomposition $\cO_K\otimes_{\mathbb{Z}}\mathbb{Z}_p=\bigoplus_{i=1}^r (\cO_{K_{\mathfrak{P}_i}}\oplus \cO_{K_{\mathfrak{P}^c_i}})$) can be identified with \[\cO_{K_{\mathfrak{P}_i}}^n\oplus \cO_{K_{\mathfrak{P}^c_i}}^n.\] Then the determinant condition implies that \[\cL \simeq \bigoplus_{i=1}^r(\cO_{K_{\mathfrak{P}_i}}^{\siga_{\tau_i}}\oplus \cO_{K_{\mathfrak{P}^c_i}}^{\sigb_{\tau_i}}),\] where $\tau_i$ is a place inducing $\mathfrak{P}_i$.  Note that there is a natural decomposition $\cL=\cL^+\oplus \cL^-$, coming from the splitting $p=w\cdot w^c$.  Also note that $\Levin\left(\mathbb{Z}_p\right)\cong\prod_{i=1}^r \left( GL_{\siga_{\tau_i}}( \cO_{K_{\mathfrak{P}_i}})\times GL_{\sigb_{\tau_i}}( \cO_{K_{\mathfrak{P}^c_i}})\right)$ can be identified with the $\cO_K$-linear automorphism group of $\cL$, 
	which induces a natural action of $\Levin$ on the dual $\cL^\vee$ of $\cL$ (by precomposing with the inverse) and on all other spaces defined in terms of $\cL$. 
\end{remark}

Now, we introduce the {\it Igusa tower} over the component $\Sord$.
For $n \in \bN$, consider the functor 
\begin{align*}
\mathrm{Ig}^{\mathrm{ord}}_{n}: \left\{\mathrm{Schemes}/\Sord\right\}\rightarrow\left\{\mathrm{Sets}\right\}
\end{align*}
that takes an $\Sord$-scheme $S$ to the set of $\OK$-linear closed immersions
\begin{equation} \label{equation-Igusa-structure}
	\cL \otimes_{\mathbb{Z}}\mu_{p^n} \hookrightarrow \cA_S[p^n],
\end{equation}
 where $\cA_S:=\cA_{\Sord}\times_{\Sord}S$. This functor is representable by an $\Sord$-scheme, which we also denote by $\mathrm{Ig}^{\mathrm{ord}}_{n}$. Let $\Witt_m:=\Witt/p^m\Witt$, and  for each $m\in\mathbb{Z}_{\geq 1}$, define $\Sord_m := \Sord\times_{\Witt}\Witt_m$. Then $\Ign_{n,m}:=\Ign_n \times_\Witt \Witt_m$ is a scheme over $\Witt_m$, whose functor 
takes an $\Sord_m$-scheme $S$ to the set of $\OK$-linear closed immersions \[\cL \otimes_{\mathbb{Z}}\mu_{p^n} \hookrightarrow \cA_S[p^n].\] 
For each $n\geq 1$, $\mathrm{Ig}^{\mathrm{ord}}_{n,m}$ is a finite \'etale and Galois covering of $\Sord_m$ whose Galois group is the group of $\cO_K$-linear automorphisms of $\cL^\vee/p^n\cL^\vee$. (See Section 8.1.1 of~\cite{hida} for a discussion of representability and of the fact that these Igusa varieties are finite \'etale covers of $\Sord_m$: the key point is that $\cL^\vee/p^n\cL^\vee$ is an \'etale sheaf.)

We also define the formal scheme $\Ig_n$ to be the formal completion of $\Ign_n$ along the special fiber $\Sord_{\fpb}$, i.e. as a functor from \{$\Sord$-schemes on which $p$ is nilpotent\} to \{Sets\}, we have $\Ig_n=\varinjlim_{m}\Ign_{n,m}$. 

This formal scheme is a finite \'etale and Galois cover of $\mathfrak{S}^{\mathrm{ord}}$, the formal completion of $\Sord$ along its special fiber. As we let $n$ vary, we obtain a tower of finite \'etale coverings of $\mathfrak{S}^{\mathrm{ord}}$, called the \emph{Igusa tower}. The Galois group of the whole Igusa tower over $\mathfrak{S}^\mathrm{ord}$ can be identified with $\Levin(\mathbb{Z}_p)$. 
 
The inverse limit of formal schemes $\mathfrak{Ig}^{\mathrm{ord}}_n$ also exists as a formal scheme, which we denote by $\mathfrak{Ig}^{\mathrm{ord}}$.  The point is that $(\mathfrak{Ig}^{\mathrm{ord}}_n)_{n\in\mathbb{Z}_{\geq 1}}$ is a projective system of formal schemes, with affine transition maps, so the inverse limit exists in the category of formal schemes.  (See, for example, \cite[Proposition D.4.1]{fargues}.)   
This is a pro-finite \'etale cover of $\mathfrak{S}^{\mathrm{ord}}$, with Galois group $\Levin(\mathbb{Z}_p)$.

We now give a different way of thinking about the Igusa tower. Let $\mathfrak{A}_{\mathfrak{S}^{\mathrm{ord}}}$ be the universal abelian variety over $\mathfrak{S}^\mathrm{ord}$. For $p$-divisible groups over $\mathfrak{S}^{\mathrm{ord}}$ there is a connected-\'etale exact sequence, so it makes sense to define the connected part $\mathfrak{A}_{\mathfrak{S}^{\mathrm{ord}}}[p^\infty]^\circ$ of $\mathfrak{A}_{\mathfrak{S}^{\mathrm{ord}}}[p^\infty]$. Then the formal completion $\mathfrak{Ig}^{\mathrm{ord}}_n$ can be identified with the formal scheme $\mathrm{Isom}_{\mathfrak{S}^{\mathrm{ord}}}(\cL \otimes_{\mathbb{Z}}\mu_{p^n}, \mathfrak{A}_{\mathfrak{S}^{\mathrm{ord}}}[p^n]^\circ)$. Using the duality induced by the Hermitian pairing on $L_p$ and by $\lambda$ on $\mathfrak{A}_{\mathfrak{S}^{\mathrm{ord}}}[p^\infty]$ (and noting that duality interchanges the connected and \'etale parts), we can further identify $\mathfrak{Ig}^\mathrm{ord}_n$ with the formal scheme $\mathrm{Isom}_{\mathfrak{S}^{\mathrm{ord}}}\left(\cL^{\vee}/p^n\cL^{\vee}, \mathfrak{A}_{\mathfrak{S}^{\mathrm{ord}}}[p^n]^\mathrm{\et}\right)$. This is finite \'etale over $\mathfrak{S}^{\mathrm{ord}}$.

\subsubsection{Irreducibility}

In this section, we show that the Igusa tower $\{\Ig_n\}_{n \in \bN}$, or equivalently $\{\mathrm{Ig}_n^{\mathrm{ord}}\}_{n \in \bN}$, is not irreducible, but we also sketch how one can pass to a partial $SU$-tower that is irreducible.

As explained above, $\Ig_n$ can be identified with 
\begin{align}\label{Igusastr}
\mathrm{Isom}_{\mathfrak{S}^\mathrm{ord}}\left(\cL^{\vee}/p^n\cL^{\vee}, \mathfrak{A}_{\mathfrak{S}^{\mathrm{ord}}}[p^n]^\mathrm{\et}\right).
\end{align}
Such an isomorphism of sheaves on $\mathfrak{S}^\mathrm{ord}$ induces an isomorphism of the top exterior powers of these sheaves, so there is a morphism \[\det: \mathrm{Isom}_{\mathfrak{S}^\mathrm{ord}}\left(\cL^{\vee}/p^n\cL^{\vee}, \mathfrak{A}_{\mathfrak{S}^{\mathrm{ord}}}[p^n]^\mathrm{\et}\right)\to \mathrm{Isom}_{\mathfrak{S}^\mathrm{ord}}\left(\wedge^{\mathrm{top}}(\cL^{\vee}/p^n\cL^{\vee}), \wedge^{\mathrm{top}}(\mathfrak{A}_{\mathfrak{S}^{\mathrm{ord}}}[p^n]^\mathrm{\et})\right).\]
 Hida \cite{hidairreducibility} shows that the sheaf $\wedge^{\mathrm{top}}(\mathfrak{A}_{\mathfrak{S}^{\mathrm{ord}}}[p^n]^\mathrm{\et})$ on $\mathfrak{S}^{\mathrm{ord}}$ is isomorphic to the constant sheaf $\cO_K/p^n\cO_K$. This gives an isomorphism \[\mathrm{Isom}_{\mathfrak{S}^\mathrm{ord}}\left(\wedge^{\mathrm{top}}(\cL^{\vee}/p^n\cL^{\vee}), \wedge^{\mathrm{top}}(\mathfrak{A}_{\mathfrak{S}^{\mathrm{ord}}}[p^n]^\mathrm{\et})\right)\toisom (\cO_K/p^n\cO_K)^\times\] and shows that the full Igusa tower $\{\mathfrak{Ig}^{\mathrm{ord}}_n\}_{n\in \mathbb{N}}$ is not irreducible.

As $n$ varies, the determinant morphisms above are compatible. 
Let $\mathfrak{Ig}^{\mathrm{ord},SU}$ be the inverse image of $(1)_{n \in \bN} \in ((\cO_K/p^n\cO_K)^\times)_{n \in \bN}$ under $\det$. 
\begin{theorem}(Hida) $\mathfrak{Ig}^{\mathrm{ord},SU}$ is a geometrically irreducible component of $\Ig$.
\end{theorem}
\begin{proof} (sketch) One proof of this statement can be found in~{\cite[Sections 3.4-3.5]{hidairreducibility}}; we merely sketch the argument here. 

It is {sufficient to show} that $\mathfrak{Ig}^{\mathrm{ord},SU}_n$ is irreducible {for each $n$}. This is an \'etale cover of $\mathfrak{S}^{\mathrm{ord}}$, the formal completion of the smooth irreducible variety $\Sord$ over $\Witt$ along its special fiber. One of Hida's strategies {(in, for example, \cite[Section 3.4]{hidairreducibility})} for proving the irreducibility of the \'etale cover in this situation is {to consider} a compatible group action of a product $\cG_1\times \cG_2$ on $\mathfrak{Ig}^{\mathrm{ord},SU}_n$ and $\mathfrak{S}^{\mathrm{ord}},$
in such a way that $\cG_1\subset \mathrm{Aut}(\mathfrak{S}^{\mathrm{ord}})$ fixes and $\cG_2\subset \mathrm{Aut}(\mathfrak{Ig}^{\mathrm{ord},SU}_n/\mathfrak{S}^{\mathrm{ord}})$ acts transitively on the connected components of $\mathfrak{Ig}^{\mathrm{ord},SU}_n$. The group $\cG_1$ can be identified with the finite adelic points $SU(\mathbb{A}^\Sigma)$ away from certain bad places $\Sigma$ containing $p$. This group will not have any finite quotient and therefore will preserve the connected components of the Igusa tower. The group $\cG_2$ can be identified with the Levi subgroup $\Levi_1(\mathbb{Z}_p):= \Levin(\mathbb{Z}_p) \cap SU(\mathbb{Z}_p)$, where $\Levin(\mathbb{Z}_p)$ is as in Remark~\ref{explicitdecomposition}. The action of $\Levi_1(\mathbb{Z}_p)$ on the connected components is transitive, since $\mathfrak{Ig}^{\mathrm{ord},SU}_n/\mathfrak{S}^{\mathrm{ord}}$ is {a} $\Levi_1(\mathbb{Z}_p/p^n\mathbb{Z}_p)$-torsor via the action on the Igusa level structure.

Following Hida's argument, we choose a base point $x_0$ on the Igusa tower.
Hida considers the group $\cT_{{x_0}}$ generated by $\cG_1$ and the stabilizer of $x_0$ in $\cG_1\times\cG_2$.  In \cite[Section 3.5]{hidairreducibility}, he shows one can choose the point $x_0$ such that $\cT_{{x_0}}$ is dense in $\cG_1\times \cG_2$, which amounts to choosing a point whose stabilizer has $p$-adically dense image in $\Levi_1(\mathbb{Z}_p)$. This density is obtained as a by-product of the fact that the abelian variety with structures corresponding to the point $x_0$
has many extra endomorphisms over $\mathbb{Q}$ and the fact that $\Levi_1(\mathbb{Z}_p)\cap SU(\mathbb{Q})$ is dense in $\Levi_1(\mathbb{Z}_p)$.  On one hand, $\cT_{{x_0}}$ is dense in a group acting transitively on the connected components of $\mathfrak{Ig}^{\mathrm{ord},SU}_n$; on the other hand, $\cT_{{x_0}}$ fixes the connected components by definition. This shows that $\mathfrak{Ig}^{\mathrm{ord},SU}_n$ has only one connected component to start with and, therefore, that it is irreducible.  {In fact, Hida shows that one can take $x_0$ to be a CM point; this is the entire subject of \cite[Section 3.5]{hidairreducibility}.}

\end{proof}

\subsection{$p$-adic automorphic forms}\label{padicaut-section}

\newcommand{\frakIg}{\mathfrak{Ig}^\mathrm{ord}}
\renewcommand{\Ig}{\mathrm{Ig}^\mathrm{ord}}
In order to define $p$-adic automorphic forms, we define the global sections
	$$V_{n,m} = H^0(\Ig_{n,m},\cO_{\Ig_{n,m}}),$$
and let $V_{\infty,m}:=\varinjlim_n V_{n,m}$ and $V:=V_{\infty,\infty}:=\varprojlim_m V_{\infty,m}$.

The space $V$ is endowed with a left action of $\Levin(\zz_p)$, $f\mapsto g\cdot f$, induced by the natural right action of $g \in \Levin (\zz_p)$ on the Igusa tower by $g\cdot f :=g^*(f)=f\circ g$.   For any point $x=(x_n) \in (\Igusa_n(\Witt))_{n \in \bN}$, where $x_{n+1}$ maps to $x_n$ under the projection, if $\underline{A}=\underline{\cA}_{x}$ denotes the associated abelian scheme over $\Witt$ endowed with additional structures, and $\iota_x:\mu_{p^\infty}\otimes \cL\hookrightarrow A[p^\infty]$ denotes the Igusa structure of infinite level on $A$, then for any $g\in\Levin(\zz_p)$ the image of $x$ under the morphism $g $ is  the point $x^g=(x^g_n) \in (\Igusa_n(\Witt))_{n \in \bN}$ corresponding to the data of the abelian scheme {$\underline{A}$ (with the associated additional structures)}, together with the Igusa structure of infinite level $\iota_{x^g}=\iota_x\circ (1\otimes g)$.

\begin{defi}\label{defi-pforms}
We call 
$V^N :=V_{\infty,\infty}^{N(\bZ_p)}$ the {\it space of $p$-adic automorphic forms}. 
\end{defi}
It is worth noting that taking invariants by $N(\bZ_p)$ commutes with both direct and inverse limits, thus we could define $V^N$ also as $\varprojlim_m \varinjlim_n V_{n,m}^{N}$ with $V_{n,m}^{N}:=V_{n,m}^{N(\bZ_p)}$.

\begin{remark}\label{rmk-k2}
In Definition \ref{defi-pforms}, we have defined $p$-adic automorphic forms over the non-compactified Shimura variety. 
Typically, $p$-adic automorphic forms are defined over the compactified Shimura variety by constructing sheaves on any toroidal compactification of $\cM$ which descend to the minimal compactification of $\cM$ and are canonically identified with $\cE_{\cpct,\rho}$ when restricted to $\cM$.  (See \S8.3.5 of \cite{lan4}.) For the present paper, we are interested in local properties of automorphic forms, namely their local behavior at ordinary CM points.  Thus, compactifications have no bearing on the main results of this paper.  Furthermore, by Koecher's Principle (see Remark \ref{remark-Koecher}), so long as we are not in the case in which both $\Sigma$ consists of only one place and the signature is $(1, 1)$, both definitions agree.
\end{remark}

\subsubsection{Comparison of automorphic forms and $p$-adic automorphic forms.}\label{lcan}
We now construct an embedding of the global sections of automorphic vector bundles on the connected component $\Sord$ of the ordinary locus into $V^N$, our newly constructed space of $p$-adic automorphic forms.

Let $n\geq m$. Recall that each element $f\in V_{n, m}^N$ can be viewed as a function
\begin{align}\label{pformasrule-equ}
\left(\uA, j\right)\mapsto f\left(\uA, j\right)\in \Witt_m,
\end{align}
where $\uA$ consists of an abelian variety $A/\Witt_m$ together with a polarization, an endomorphism, and a level structure, and an $\cO_K$-linear isomorphism $j: A[p^n]^{\et}\isomto (\bZ/p^n\bZ)\otimes\lattice^\vee$.

Recall that each element $g$ in $\Levin(\bZ/p^n\bZ)$ acts on elements $f$ in $V_{n,m}$ via 
\begin{align*}
(g\cdot f)(\uA, j) := f(\uA, g j).
\end{align*}

We may view each element $f\in H^0\left(\Sord_{m}, \Ekap\right)$ as a function
\begin{align*}
(\uA, \lambdaiso)\mapsto f(\uA, \lambdaiso)\in \Ind_{B^-}^\Levin\left(-\kappa\right)_{\Witt_m},
\end{align*}
where $\uA$ is as in Equation \eqref{pformasrule-equ}, $\lambdaiso \in \Isom_{\cO_{\Sord_m}}(\bZ/p^n\bZ \otimes \cL^{\vee},\uo_{\cA/\Sord_m})$ and $\bA^1$ is the affine line, satisfying $f\left(\uA, g\lambdaiso\right) = \rho_\kappa\left(\left({ }^tg\right)^{-1}\right)f\left(\uA, \lambdaiso\right)$ for all $g\in\Levin\left(\Witt_m\right)$.  (Compare with Definition \ref{algauto-defi1}.) 
Equivalently, we may view $f$ as a function
\begin{align*}
	(\uA, \lambdaiso)\mapsto f(\uA, \lambdaiso)\in \Ind_B^\Levin\left(\kappa\right)_{\Witt_m} = \{ f: \Levin_{\Witt_m}/N_{\Witt_m} \rightarrow \bA^1_{\Witt_m} : f(ht) = \kappa(t)f(h), t \in T\}
\end{align*}
that satisfies $f\left(\uA, g\lambdaiso\right) = \rho\left(g\right)f\left(\uA, \lambdaiso\right)$ for all $g\in\Levin\left(\Witt_m\right)$, where
the action of $\Levin$ on $\Ind_B^{\Levin}(\kappa)_{\Witt_m}$ via $\rho$ is given by 
	$$\Levin \ni h: f(x) \mapsto \rho(h)f(x) = f(h^{-1}x).$$

For $n\geq m$ and $A$ ordinary over $\Witt_m$, we have 
\begin{align*}
\Lie A= \Lie A[p^n]^\circ.
\end{align*}
Thus, for the universal abelian variety $\Auniv$, we have (compare \cite[Section~3.3-3.4]{kaST})
\begin{align}\label{LieA-equ}
\uo_{\Auniv/\Sord_m} = (\Lie \Auniv)\dual = (\Lie \Auniv[p^n]^\circ)\dual \cong (T_p(\Auniv[p^n]^{\et})\dual)\dual \cong \Auniv[p^n]^{\et}\otimes\cO_{\Sord_{m}}.
\end{align}
By Equation \eqref{LieA-equ},
there is a canonical isomorphism 
\begin{align*}
\uo_{\cA/\Sord_m}\isomto \Auniv[p^n]^{\et}\otimes\cO_{\Sord_{m}}.
\end{align*}
Thus, we may view $j: A[p^n]^{\et}\isomto (\bZ/p^n\bZ)\otimes\lattice^\vee$ as an element of $\Isom_{\cO_{\Sord_m}}(\bZ/p^n\bZ \otimes \cL^{\vee},\uo_{\cA/\Sord_m})$ and each element $f\in H^0\left({\Sord_{m}}, \Ekap\right)$ yields a function
\begin{align*}
\left(\uA, j \right)\mapsto f\left(\uA, j \right)\in \Ind_B^\Levin\left(\kappa\right)_{\Witt_m},
\end{align*}
where $\uA$ and $j$ are as in Equation \eqref{pformasrule-equ}, satisfying $f\left(\uA, gj\right) = \rho\left(({}^tg)^{-1}\right)f\left(\uA, j\right)$ for $g \in \Levin(\bZ/p^n\bZ)$ (due to the action of $g$ on $\lambdaiso$ by precomposition with ${}^tg$, {as described in} Section \ref{classical-aut-forms}).

As explained in \cite[Section 8.1.2]{hida}, there is a unique (up to $\Witt$-unit multiple) $N$-invariant element $\ellcan\in \left(\Ind_B^\Levin(\kappa)_\Witt\right)\dual = \Hom_{\Witt}(\Ind_B^\Levin(\kappa)_\Witt,\Witt)$.  The element $\ellcan$ generates $\left(\left(\Ind_B^\Levin(\kappa)_\Witt\right)\dual\right)^N$.  We may normalize $\ellcan$ so that it is evaluation at the identity in $\Levin$.

Now, we define a map of functions that turns out to be a map of global sections
\begin{align*}
\Psi_{n,m}: H^0\left({\Sord_{m}}, \Ekap\right)&\quad \rightarrow \quad V_{n,m}^N[\kappa]\\
f&\quad \mapsto \quad \tilde{f}: (\uA, j) \mapsto \ellcan(f(\uA, j)),
\end{align*}
where $V_{n,m}^N[\kappa]$ denotes the $\kappa$-eigenspace of the torus.
Note that for each $b\in B(\bZ/p^n\bZ)$,
\begin{align*}
(b\tilde{f})(\uA, j) &= \tilde{f}(\uA, bj)\\
 &= \ellcan(f(\uA, bj))\\
& = f(\uA, bj)(1)\\
&=\rho\left(({}^tg)^{-1}\right) f(\uA, j)(1)\\
& = f(\uA, j)({}^tb)\\
&= \kappa(b)f(\uA, j)(1)=\kappa(b)\ellcan(f(\uA, j)),
\end{align*}
for all $\uA$ and $j$ as in Equation \eqref{pformasrule-equ}.
So $\tilde{f}$ is indeed in $V_{n,m}^N[\kappa] \subset V_{n,m}$.

We therefore have a map
\[\Psi_\kappa:H^0\left({\Sord}, \Ekap\right)\rightarrow V^N[\kappa],\]
where $V^N[\kappa]$ denotes the $\kappa$-eigenspace of the torus,
{and we define 
	$$\Psi=\oplus_{\kappa}\Psi_\kappa: \oplus_\kappa H^0(\Sord,\Ekap) \ra V^N.$$ }

This map yields an embedding.
\begin{proposition}{\cite[Section 8.1.3, p. 335]{hida}}\label{ten} 
	The map $\Psi$ is injective.  
\end{proposition}

\begin{remark}
While it is not the subject of this paper, it is natural to ask about results on the density of classical automorphic forms within the space of $p$-adic automorphic forms. Because our main theorems do not use such density results in our arguments, we refer the reader to \cite[Chapter 8]{hida} and \cite[Lemma 6.1]{HLTT}. Additionally, it is also true by \cite[Theorem IV.3.1]{scholze} that classes in the completed cohomology of Shimura varieties are also $p$-adically interpolated from classical automorphic forms. This statement is stronger than all previous results, since it also applies to torsion classes which contribute to completed cohomology. 
\end{remark}

\section{Serre-Tate expansions}\label{ST-section} 
The goal of this section is to establish a $p$-adic analogue of the $q$-expansion principle for automorphic forms as a consequence of Hida's irreducibility result for the Igusa tower.

Classically, $q$-expansions arise by localization at a cusp, i.e., the $q$-expansion of a scalar-valued form $f$ is the image of $f$ in the complete local ring at the cusp, regarded as a power series in $q$, for $q$ a canonical choice of the local parameter at the cusp.  In these terms, the $q$-expansion principle states that localization is injective, and it is an immediate consequence of the fact that the space is connected. Alternatively, when working over the whole Shimura variety, it becomes necessary to choose a cusp on each connected component, and consider all localizations at once.

In this paper, we work over a connected component of $\cM$, and we replace cusps with integral ordinary CM points (i.e. points of $\Sord$ defined over $\Witt$ corresponding to abelian varieties with complex multiplication).
The crucial observation is that given an integral ordinary CM point $x_0$, the choice of a lift $x$ of $x_0$ to the Igusa tower uniquely determines a choice of \ST local parameters at $x_0$, i.e. $x$ defines an isomorphism of the $p$-adic completion of the complete local ring at $x_0$ with a power series ring over $\Witt$.  We call the power series corresponding to the localization at $x$ of an automorphic form its $t$-expansion, for $t$ denoting the \ST local parameters.

{By} abuse of notation, we denote by $g:\Igusa\rightarrow \Igusa$ the action of $g\in\Levin(\zz_p)$ on the Igusa tower described in Section \ref{padicaut-section}, and we write $\otimes$ in place of $\otimes_{\zz_p}$ for the tensor product over $\zz_p$.
 
\subsection{Localization}\label{local-section}
Let ${\bar x}_0\in\Sord(\fpb)$ be a geometric point, and let $x_0\in \Sord(\Witt)$ be any integral lift of ${\bar{x}}_0$. (Without loss of generality, we may chose $x_0$ to be a CM point, or even the canonical CM lift of ${\bar{x}}_0$; see Remark \ref{CMlift}.)  We write ${\Sord}^\wedge_{x_0}$  for the formal completion of $\Sord$ at $x_0$. Then \[{\Sord}^\wedge_{x_0}={\rm Spf}( \Ring_{\Sord, x_0})\] where $\Ring_{\Sord, x_0}$ is a $p$-adically complete local ring.

More explicitly, $\Ring_{\Sord, x_0}$ can be constructed as follows. For each $m\geq 1$, we write $x_{0,m}$ for the reduction of $x_0$ modulo $p^m$, regarded as a point of $\cS_m^{\rm ord}:=\Sord\times_\Witt \Witt/p^m\Witt$ (in particular, ${\bar{x}}_0=x_{0,1}$). Let $\cO^{\wedge}_{\cS^{\rm ord}_m,x_0}$ denote the completed local ring of $\cS^{\rm ord}_m$ at $x_{0,m}$.
Then, the  local ring $\Ring_{\Sord,x_0}$ can be identified with $\varprojlim_m \cO^\wedge_{\cS^{\rm ord}_m,x_0}$. Alternatively,  $\Ring_{\Sord,x_0}$ can also be identified with $\cO^{\wedge}_{\Sord,{\bar{x}}_0}$, the completed local ring of $\Sord$ at ${\bar x}_0$.

Let $x\in \Igusa(\Witt)$ denote a compatible system of integral points $x=\left(x_n\right)_{n\geq 0}$ on the Igusa tower $\{\Ig_n\}_{n\geq 0}$ above $x_0$; i.e. for each $n\geq 0$, $x_n$ is an integral point in $\Igusa_n$, $\Igusa_0=\Sord$, with $x_n$ mapping to $x_{n-1}$ under the natural projections.  Given the point $x_0$ of $\Sord$, the choice of a point $x$ of $\Igusa$ lying above $x_0$ is equivalent to the choice of an Igusa structure of infinite level (i.e.  of compatible closed immersions as in Equation \eqref{equation-Igusa-structure} 
 on the corresponding ordinary abelian variety.

For all $m$, as $n$ varies, the natural finite \'etale projections $j=j_{n,m}:\Igusa_{n, m}\rightarrow \cS^{\rm ord}_m$ induce  a compatible system of isomorphisms
$$j^*_{x}: \cO^{\wedge}_{\cS^{\rm ord}_m,x_0}\isomto\cO^{\wedge}_{\Igusa_{n ,m},x_n}$$
 which allow us to canonically identify $\cO^\wedge_{\Ig,x}:=\varprojlim_m \varinjlim_n \cO^{\wedge}_{\Igusa_{n,m},x_{n}}$ with $\Ring_{\Sord,x_0}$.

\begin{remark}
Given $x_0\in \Sord (\Witt)$, let $\bar{x}_0\in\Sord(\fpb)$ denote its reduction modulo $p$. We observe that, as a consequence of the fact that the morphisms in the Igusa tower are \'etale, the reduction modulo $p$ gives a canonical bijection between the points on the Igusa tower above $x_0$ and those above ${\bar{x}}_0$.   Moreover,  if  we denote by $\bar{x}\in \Ig(\fpb)$, $\bar{x}=(\bar{x}_n)_{n\geq 0}$, the reduction of $x$ modulo $p$, then the previous isomorphisms agree with the compatible system of isomorphisms
 $ j_{n,{\bar x}_n}^*: \cO^{\wedge}_{\Sord,{\bar{x}}_0}\rightarrow \cO^{\wedge}_{\Igusa_{n},\bar{x}_{n}} $.
\end{remark}

\begin{defn}
Let $x\in\Igusa(\Witt)$, lying above $x_0\in\Sord(\Witt)$. We define \[\loc_x:V\to \Ring_{\Sord,x_0}\] as the localization at $x$ composed with ${j^*_x}^{-1}$. 

\end{defn}

By abuse of language,  we will still refer to $\loc_x(f)\in \Ring_{\Sord,x_0}$ as the localization of $f$ at $x$, for all $f\in V$.  Furthermore, with abuse of notation,  we will still denote  by $\loc_x$ the restriction of $\loc_x$ to $V^N$ (resp. $V^N[\kappa]$, for any weight $\kappa$).

We compare localizations at different points of the Igusa tower. Recall that given a point $x \in \Igusa(\Witt)$, the set of all points of the Igusa tower above $x_0=j(x)\in\Sord(\Witt)$ is a principle homogeneous space for the action of $\Levin(\zz_p)$.

\begin{lemma} \label{Leviaction}  
Let $x\in\Igusa(\Witt)$.
For any $g\in\Levin(\zz_p)$ and $f\in V$, we have \[\loc_{x^g}(f)=\loc_{x}\left(g\cdot f \right).\]

In particular, if $g\in T(\zz_p)$ and $f\in V^N[\kappa]$, for a weight $\kappa$, we have \[\loc_{x^g}(f)=\kappa(g)\loc_{x}\left(f\right).\] 
\end{lemma}

\begin{proof}
For all $g\in\Levin(\zz_p)$ and $x\in \Igusa(\Witt)$, the morphism $g:\Igusa \rightarrow \Igusa$ induces an isomorphism of complete local rings $g^*:\cO^\wedge_{\Igusa, x^g}\rightarrow \cO^\wedge_{\Igusa, x}$, $\phi\mapsto \phi\circ g$. Then, the first statement follows from the definition given the equality ${j^*_{x^g}}^{-1}={j^*_{x}}^{-1}\circ g^*$. The second statement is an immediate consequence of the first one, given that for all $f\in V^N[\kappa]$ and $g\in T(\zz_p)$ we have $g\cdot f=\kappa(g)f$.
\end{proof}

\begin{prop}\label{propq}
Let $x\in\Igusa (\Witt)$. For any weight $\kappa$, the map ${\loc}_x:V^N[\kappa]\rightarrow \Ring_{\Sord,x_0}$ is injective.
\end{prop}
\begin{proof} 
Were the Igusa covers $\Ig_n$ irreducible over $\Sord$, the statement would immediately follow.  As it happens $\Ig_n$ is not irreducible, thus a priori  the vanishing under the localization map $\loc_x$ only implies the vanishing on the connected component containing $x$. Yet, as the torus $T(\zz_p)\subset \Levin(\zz_p)$ acts transitively on the connected components of $\Ig$ over $\Sord$, it suffices to prove that for all $f\in V^N[\kappa]$, the identity $\loc_x(f)=0$ implies $\loc_{x^g}(f)=0$ for all $g\in T(\zz_p)$. The last statement follows immediately from the second part of Lemma \ref{Leviaction}.
\end{proof}

We note that for a general function $f\in V$ the above statement is false. Yet, by the same argument, Lemma \ref{Leviaction} implies the following weaker statement for all $f\in V$.

\begin{prop}\label{generalq}
Let $x\in \Igusa (\Witt)$, lying above $x_0\in\Sord (\Witt)$. For any $f\in V$,  if $\loc_{x}(g\cdot f)\in \Ring_{\Sord,x_0}$ vanishes for all $g\in T(\zz_p)$, then $f=0$.
\end{prop}

\subsection{Serre-Tate coordinates}\label{\ST thy}
It follows from the smoothness of $\Sord$ that for any $\Witt$-point $x_0$ the ring $\Ring_{\Sord,x_0}$ is (non-canonically) isomorphic to a power series ring over $\Witt$.
The goal of this section is to explain how Serre-Tate theory implies that for $x$ a lift of $x_0$ to $\Ig$,
the ring $\cO^\wedge_{\Ig,x}$ is canonically isomorphic to a ring of power series over $\Witt$.

We recall Serre-Tate theory following \cite{kaST}.
The first theorem describes the deformation space of an ordinary abelian variety, and the second explains how to address the lifting of additional structures (such as a polarization and extra endomorphisms). As an application we deduce a description of the ring $\Ring_{\Sord,x_0}$, for any $x_0\in \Sord(\Witt)$.

We introduce some notation.
Let $A$ be an ordinary abelian variety over $\fpb$, of dimension $g$.  The {\em physical Tate module} of $A$, $T_pA(\fpb)$, is the Tate module of the maximal \'etale quotient of $A[p^\infty]$, i.e $$T_pA(\fpb)=\varprojlim_nA[p^n](\fpb)=\varprojlim_nA[p^n]^{\et}(\fpb).$$ As $A$ is ordinary, $T_pA(\fpb)$ is a free $\bZ_p$-module of rank $g$.  Like above, we denote by $A^\vee$ the dual abelian variety, and we denote by $T_pA^\vee(\fpb)$ the physical Tate module of $A^\vee$.

Let $R$ be an Artinian local ring, with residue field $\fpb$. We denote by $\mathfrak{m}_R$ the maximal ideal of $R$. A {\em lifting} (or {\em deformation}) of $A$ over $R$ is a pair $(\cA/R,j)$,
consisting of an abelian scheme $\cA$ over $R$, together with an isomorphism $j:\cA\otimes_R \fpb\rightarrow A$.
By abuse of notation we sometimes simply write $\cA/R$ for the pair $(\cA/R,j)$.  To each lifting $(\cA/R,j),$
as explained in \cite[Section 2.0, p. 148]{kaST}, Serre and Tate associated a $\bZ_p$-bilinear form
$$q_{\cA/R} :T_pA(\fpb)\times T_pA^\vee(\fpb)\rightarrow \hat{\bG}_m(R)=1+\mathfrak{m}_R.$$

\begin{theorem} [\cite{kaST} Theorem 2.1, p. 148]\label{ST-theorem}
Let the notation be as above.
\begin{enumerate}
\item The map $(\cA/R,j)\mapsto q_{\cA/R}$  is a bijection from the set of isomorphism classes of liftings of $A$ over $R$ to the group
${\rm Hom}_{\bZ_p}(T_pA(\fpb)\otimes T_pA^\vee(\fpb),\hat{\bG}_m(R))$.
\item
The above construction defines an isomorphism of functors between the deformation space $\cM_{A/\fpb}$ and
${\rm Hom}_{\bZ_p}(T_pA(\fpb)\otimes T_pA^\vee(\fpb),\hat{\bG}_m)$.
\end{enumerate}
\end{theorem}

In the following we refer to the above isomorphism as the {\it \ST isomorphism}.

Let $A$ and $B$ be ordinary abelian varieties over $\fpb$, and let $f:A\rightarrow B$ be an $\fpb$-isogeny. We write $f^\vee: B^\vee\rightarrow A^\vee$ for the dual isogeny. A theorem of Drinfield (\cite[Lemma 1.1.3, p.141]{kaST})  proves that for any Artinian local ring $R$, and pair of  liftings  $(\cA/R,j_{\cA})$,$(\cB/R,j_{\cB})$ of $A,B$ respectively, if there exists an isogeny $\phi:\cA\rightarrow \cB$ lifting $f$ (i.e. satisfying $f=j_{\cB}\circ (\phi\otimes 1_{\fpb}) \circ j_{\cA}^{-1}$), then $\phi$ is unique. Yet, in general such a lifting of $f$ will not exists.  Theorem \ref{thmcon} gives a necessary and sufficient condition for the existence of $\phi$ in terms of the \ST isomorphism.

\begin{theorem}[\cite{kaST} Theorem 2.1, Part 4,  p.149]\label{thmcon}
Let the notation be as above.  Given $\cA$ and $\cB$ lifting $A$ and $B$, respectively, over an Artinian local ring $R$.  A morphism $f:A\rightarrow B$ lifts to a morphism $\cA\rightarrow \cB$ if and only if $q_{\cA/R}\circ (1\times f^\vee)=q_{\cB/R}\circ (f\times 1)$.
\end{theorem}

We apply the above results in our setting, adapting to our context the arguments in \cite[Sections 8.2.4 and 8.2.5]{hida}. (In loc. cit,  Hida deals respectively with the cases of Siegel varieties and of the unitary Shimura varieties over a quadratic imaginary field in which $p$ splits.)

Let $\cO_{\cmfield, p} := \cO_K\otimes_\ZZ\ZZ_p$.  Recall that under our assumptions $p$ splits completely in $K$. I.e.,  we have $\cO_{K,p}=\prod_{i=1}^r \cO_{K_{\fP_i}}\times \prod_{i=1}^r \cO_{K_{\fP^c_i}}$, where  $\cO_{K_{\fP_i}}\simeq\zz_p$ and $\cO_{K_{\fP^c_i}}\simeq\zz_p$ for all $i$.

Let $\bar{x}_0\in \Sord(\fpb)$ be an ordinary point, and $\underline{A}:=\underline{\cA}_{\bar{x}_0}$ be the associated abelian variety over $\fpb$, together with its additional structures. Then, the physical Tate module $T_pA(\fpb)$ of $A$ is a free $\cO_{K,p}$-module, and the prime-to-p polarization $\lambda $ of $A$ induces a conjugate-linear isomorphism $T_p(\lambda):T_p(A)(\fpb)\isomto T_pA^\vee(\fpb)$.   We deduce that \[T_pA(\fpb)=\left(\oplus_{i=1}^r T_{\fP_i}A(\fpb)\right)\oplus\left(\oplus_{i=1}^r  T_{\fP^c_i}A(\fpb)\right)\] and that $T_p(\lambda)$ induces a $\zz_p$-linear isomorphism between $T_{\fP_i^c}A(\fpb)$ and $T_{\fP_i}A^\vee(\fpb)$.\

Finally, we recall that the formal completion ${\Sord}^\wedge_{\bar{x}_0}$ of $\Sord$ at $\bar{x}_0$ represents the deformation problem naturally associated with the moduli problem. This observation allows us to canonically identify ${\Sord}^\wedge_{\bar{x}_0}$  with the closed subspace of the deformation space $\cM_{A/\fpb}$ consisting of all deformations  of $A$ which are (can be) endowed with additional structures (a polarization and an $\cO_K$-action) lifting those of $A$. We deduce the following description of ${\Sord}^\wedge_{\bar{x}_0}$.

\begin{prop}  \label{STthm} Let the notation be as above.

The map $x\mapsto q_x:=(1\times T_p(\lambda))\circ q_{\cA_x}$, from ${\Sord}^\wedge_{\bar{x}_0}$ to ${\rm Hom}_{\zz_p}(T_pA(\fpb)\otimes T_pA(\fpb),\hat{\bG}_m)$, induces
 an isomorphism between 
${\Sord}^\wedge_{\bar{x}_0}$  and
$\oplus_{i=1}^r{\rm Hom}_{\zz_p}(T_{\fP_i}A(\fpb)\otimes T_{\fP_i^c}A(\fpb),\hat{\bG}_m)$.
\end{prop}

\begin{proof} 
For any  local Artinian ring  $R$ and $x\in\Sord(R)$, lifting $\bar{x}_0$,
let $q_x$ denote the $\zz_p$-bilinear form defined in the statement
\[q_x=(1\times T_p(\lambda))\circ q_{\cA_x}:T_pA(\fpb)\times T_pA(\fpb)\rightarrow \hat{\mathbb G}_m(R).\]

By Theorem \ref{thmcon} 
the additional structures on $A$ (namely, the polarization and the $\cO_K$-action respectively) lift to $\cA_x$ if and only if $q_x$  is symmetric and $c$-hermitian.

Indeed, let  ${\rm sw}:T_pA(\fpb)\times T_pA(\fpb)\rightarrow T_pA(\fpb)\times T_pA(\fpb)$ denote 
the map ${\rm sw}(v,w)=(w,v)$, for all $v,w\in T_pA(\fpb)$.  Then, given any abelian variety $\cA$ lifting $A$, $q_{{\cA}^\vee} =q_{\cA} \circ {\rm sw}$. Also, note that since the polarization $\lambda $ has degree prime-to-$p$ then the induced maps on the physical Tate modules satisfy the condition $T_p(\lambda^\vee)=T_p(\lambda)^{-1}$. 
We deduce that the polarization $\lambda$ of $A$ lifts to $\cA_x$ if and only if $q_{\cA_x}\circ (1\times T_p(\lambda^\vee))=q_{\cA_x^\vee}\circ (T_p(\lambda)\times 1)$, or equivalently  if and only if $q_x=q_x\circ {\rm sw}$.
Similarly, for all $b\in\cO_K$, let $i(b)$ denote the action on $A$. We deduce that
the action of $\cO_K$ on $A$ lifts to $\cA$ if and only if  $q_{\cA_x}\circ (1\times i(b)^\vee)=q_{\cA_x}\circ (i(b)\times 1)$, for all $b\in\cO_{K}$, or equivalently
if and only if  $q_{x}(1\times i(b^c))=q_{x}(i(b) \times 1)$, because $\lambda\circ i(b^c)=i(b)^\vee\circ \lambda$ by definition.  (Recall Section \ref{PELmoduli-section}.)
 
For all $i,j=1,\cdots ,r$, let the forms $q_{x,i,j}:T_{\fP_i}A(\fpb)\times T_{\fP_j^c}A(\fpb)\rightarrow \hat{\bG}_m(R)$ and  $q_{x,i,j,c}:T_{\fP_i^c}A(\fpb)\times T_{\fP_j}A(\fpb)\rightarrow \hat{\bG}_m(R)$ denote the restrictions of $q_x$. Then,  since $q_x$ is symmetric, for all $i,j=1,\cdots ,r$, we have $q_{x,i,j}=q_{x,j,i,c}\circ {\rm sw},$ where with abuse of notation we still write ${\rm sw}$ for its restriction $T_{\fP_i}A(\fpb)\times T_{\fP_j^c}A(\fpb)\rightarrow T_{\fP_j^c}A(\fpb)\times T_{\fP_i}A(\fpb)$.  Furthermore,  since $q_x$ is $c$-hermitian, we deduce that  for all $i\neq j$ the forms $q_{x,i,j}$ and $q_{x,i,j,c}$ necessarily vanish. Thus, \[q_x=(\oplus_{i=1}^rq_{x,i,i})\oplus (\oplus_{i=1}^rq_{x,i,i}\circ {\rm sw}).\]  In particular,  the form $q_x$ is uniquely determined by its restrictions
\[q_{x, i,i} :T_{\fP_i}A(\fpb)\times T_{\fP_i^c}A(\fpb)\rightarrow \hat{\bG}_m(R),\] for $i=1, \dots r$. To conclude we observe that any collection $(q_i)_{i=1,\dots , r}$ of $\zz_p$-bilinear morphisms, $q_i: T_{\fP_i}A(\fpb)\times T_{\fP_i^c}A(\fpb)\rightarrow \hat{\bG}_m(R)$, extends uniquely to a symmetric $c$-hermitian bilinear form $q:T_pA(\fpb)\otimes T_pA(\fpb)\rightarrow \hat{\mathbb G}_m(R)$, namely $q=(\oplus_{i=1}^rq_{i})\oplus(\oplus_{i=1}^r q_{i}\circ {\rm sw})$.
\end{proof}

\begin{remark}\label{CMlift}
As a consequence of the \ST theory, we see that any point ${\bar{x}}_0\in \Sord(\fpb)$ always lifts to a point $x_0\in\Sord(\Witt)$. In fact, it even admits a canonical lift $y_0$, namely the one corresponding to the form $q=0$.  The abelian variety $\cA_{y_0}$ is the unique deformation of $\cA_{\bar{x}_0}$ to which all endomorphisms lift. Thus, in particular, $\cA_{y_0}$ is a CM abelian variety with ordinary reduction. The point $y_0$ is called the canonical CM lift of ${\bar{x}}_0$.
\end{remark}

We finally describe how the choice of an Igusa structure (of infinite level) on $\underline{A}=\underline{\cA}_{x_0}$ determines a choice of local parameters.

We consider the $\cO_{K,p}$-modules introduced in Section \ref{Igusalevel}:  \[\cL=\cL^+\oplus \cL^-\simeq \oplus_{i=1}^r(\cO_{K_{\mathfrak{P}_i}}^{a_{\tau_i+}}\oplus \cO_{K_{\mathfrak{P}^c_i}}^{a_{\tau_i-}}),\] 
where the decomposition comes from
the splitting $p=w\cdot w^c$, and for each $i=1,\dots ,r$, $\tau_i$ is a place inducing $\mathfrak{P}_i$.   For each $i$, we write $\cL^+_{i}\subset \cL^+$ for the submodule corresponding to the place $\tau_i$,  $\cL^+_i\simeq \cO_{K_{\mathfrak{P}_i}}^{a_{\tau_i+}}$. Similarly, we define
$\cL^-_{i}\subset \cL^-$, $\cL_i^-\simeq \cO_{K_{\mathfrak{P}_i}}^{a_{\tau_i-}}$.  Then, $\cL^+=\oplus_{i=1}^r \cL^+_i$ and $\cL^-_i=\oplus_{i=1}^r \cL^-_i$.

We define the $\cO_{K,p}$-module
\[\cL^2\simeq\oplus_{i=1}^r \cL^+_i\otimes \cL^-_i\]
naturally regarded as a submodule of $\cL\otimes\cL$. \label{STthy}

\begin{prop}\label{propbeta} 
Let $x_0\in\Sord(\Witt)$.  Each point $x\in\Igusa(\Witt)$  above $x_0$  determines a unique isomorphism \[\beta_x: {\Sord}^\wedge_{\bar{x}_0}\rightarrow {\hat{\bG}}_m\otimes \cL^2.\]
In particular, for each $x$ we have an isomorphism of local rings \[\beta^*_x: \Witt[[t]]\otimes(\cL^2)^\vee\isomto\Ring_{\Sord, x_0} \, ,\]
{where $\Witt[[t]]\otimes(\cL^2)^\vee$ denotes the complete ring corresponding to the formal scheme ${\hat{\bG}}_m\otimes \cL^2$, i.e. a choice of basis of $(\cL^2)^\vee$ yields an isomorphism $\Witt[[t]]\otimes(\cL^2)^\vee \simeq \Witt[[t_j \, | \, 1 \leq j \leq \sum_{i=1}^r{a_{\tau_i+}a_{\tau_i-}}]]$.}

\end{prop}
\begin{proof}
The choice of an Igusa structure of infinite level $\iota=\iota_x:\mu_{p^\infty}\otimes \cL\hookrightarrow A[p^\infty]$ is equivalent to the choice of an $\cO_{K,p}$-linear isomorphism $T_p(\iota^\vee):T_pA(\fpb) \isomto\cL^{\vee}$.  We observe that by linearity $$T_p(\iota^\vee)(T_wA(\fpb))=(\cL^+)^{\vee} \text{ and } T_p(\iota^\vee)(T_{w^c}A(\fpb))=(\cL^-)^{\vee}.$$  More precisely, $T_p(\iota^\vee)$ induces isomorphisms
$T_{\fP_i}A(\fpb)\simeq (\cL_i^+)^{\vee}$ and $T_{\fP_i^c}A(\fpb)\simeq(\cL^-_i)^{\vee}$, for all $i=1, \dots ,r$.
The isomorphism $\beta_x$ is defined as the composition of  the \ST isomorphism in proposition \ref{STthm} with the inverses of the trivializations induced by $T_p(\iota^\vee)$, while $\beta^*_x$ is the corresponding ring homomorphism, for $t$ the canonical parameter on ${\hat{\bG}}_m$.
\end{proof}

Let $\mathbb I$ denote the identity on $\Witt[[t]]$. For all $g\in\Levin (\zz_p)$, we write ${\mathbb I}\otimes g$ for the natural left $\Witt$-linear action on $  \Witt[[t]]\otimes(\cL^2)^\vee$.

As in Lemma \ref{Leviaction}, the action of $\Levin(\zz_p)$ on the points of the Igusa tower lying  above  $x_0$ allows us to compare the above construction for different points $x$.

\begin{lemma}\label{Leviaction2} For all $x\in\Igusa(\Witt)$ and $g\in\Levin(\zz_p)$, we have 
\[\beta^*_{x^g}=\beta^*_x\circ ({\mathbb I}\otimes g).\] 
\end{lemma}

\begin{proof}
 The statement follows immediately from the definition since $\iota_{x^g}=\iota_x\circ (1\otimes g)$ .
\end{proof}
\begin{defn}
Given a point $x\in\Igusa(\Witt)$,  for any $f\in V$ we define the $t$-expansion of  $f$ at the point $x$ as \[ f_x(t):= {\beta^*_{x}}^{-1}(\loc_{x} (f)).\]
\end{defn}

\begin{prop} \label{Leviaction3}
Let $x\in\Igusa(\Witt)$. For all $g\in \Levin(\zz_p)$ and $f\in V$ we have 
\[f_{x^g}(t)= ({\mathbb I}\otimes g^{-1} ) (g\cdot f)_x(t).\] 
\end{prop}
\begin{proof}
The statement follows from Lemmas \ref{Leviaction}, \ref{Leviaction2} combined.
\end{proof}
Finally, we can state an appropriate analogue of the $q$-expansion principle for $p$-adic automorphic forms. 

\begin{theorem}[``$t$-expansion principle'' or ``\ST expansion principle'']\label{STexpprinciple-thm}
Let  $x\in\Igusa(\Witt)$. 
For any weight $\kappa$ and any  $f\in V^N[\kappa]$, the $t$-expansion $f_x(t)$  vanishes if and only if $f$ vanishes.

\end{theorem}

\begin{proof}
The statement follows from Prositions \ref{propq} and \ref{propbeta} combined.
\end{proof}

Furthermore, by combining Propositions  \ref{generalq}  and \ref{propbeta} together we deduce the following weaker statement for all $f\in V$.

\begin{prop}\label{generalt} 
Let $x\in\Igusa(\Witt)$.
For any $f\in V$,  if $(g\cdot f)_{x}(t)\in \Witt[[t]]\otimes (\cL^2)^\vee $ vanish for all $g\in T(\zz_p)$, then $f=0$.
\end{prop}

As an immediate consequence of Proposition \ref{generalt} we obtain the following corollary.

\begin{cor} \label{cong}
For each $m\in \bN$,
given two $p$-adic automorphic forms $f,f'$ of any weight $\kappa,\kappa'$ respectively, we have that  $f\equiv f' (\mod p^m)$ if and only if \[\kappa(g)f_x(t)\equiv\kappa'(g) f'_x(t)\, (\mod p^m)\] for all $g\in T(\zz_p)$.
\end{cor}

\begin{remark} 
A crucial advantage of Hida's realization of (vector-valued) automorphic forms as function on the infinite Igusa tower is that one can define congruences between forms without any restriction on their weights.
The $t$-expansion principle as stated above implies that those congruence relations can be detected by the coefficients of the associated power series, as in the classical setting.
Indeed, for $f$ a vector-valued automorphic form of weight $\kappa$, one may also consider the
localization of $f$ at a point $x_0$ of $\Sord$. Such a localization is an element in a free module $M_\kappa$ over $\Ring_{\Sord,x_0}$, with rank depending on the weight $\kappa$.  Even with canonical choices of trivializations of the modules $M_\kappa$'s, such an approach only allows to detect congruences when the two ranks agree, e.g. when the difference among the weights is parallel, i.e. equals $(n^\tau, \hdots, n^\tau)_{\tau \in \Sigma}$ in the notation of Section \ref{Section-weights}. This restriction is not necessary for the above corollary.
\end{remark}

\section{Restriction of $t$-expansions}\label{section on pullbacks}
As an example of how \ST coordinates may be used to understand certain operators on $p$-adic automorphic forms, we consider the case of the restriction map from our unitary Shimura variety to a lower dimensional unitary Shimura subvariety, i.e. the pullback map on global sections of the automorphic sheaves.

\subsection{Description of the geometry}
We start by introducing the PEL data defining this setting. We maintain the notation introduced in Section \ref{unitarygroups-section}.

Let $L=\oplus_{i=1}^s W_i$ be a self-dual decomposition of the lattice $L$. We denote by ${\langle, \rangle}_i$ the pairing on $W_i$ induced by $\langle,\rangle$ on $L$, and define
$GU_i=GU(W_i,\langle, \rangle_i)$, a unitary group of signature
$\left(\siga_{\tau,i},\sigb_{\tau,i}\right)_{\tau \in \Sigma_K}$. Then, the signatures $\left(\siga_{\tau,i},\sigb_{\tau,i}\right)_{i=1,\dots ,s}$ form a partition of the signature $\left(\siga_{\tau},\sigb_{\tau}\right)$.

We define $G'\subset \prod_iGU_i$ to be the subgroup of elements with the same similitude factor; i.e. if $\nu_i:GU_i\to \bG_m$ denote the similitude factors, then $G'=\nu^{-1}(\bG_m)$, for $\nu=\prod_i\nu_i$ and $\bG_m\subset \bG_m^s$ embedded diagonally. Then, there is a natural closed immersion $G'\rightarrow GU$ of algebraic groups which is compatible with the partition of the signature.  That is, the above data defines a morphism of Shimura data $(G',X')\rightarrow (GU,X)$, for $X$ (resp. $X'$) the $GU(\bR)$- (resp. $G'(\bR)$-) conjugacy class of  the homomorphism  $h:\mathbb{C}\to \mathrm{End}_{K\otimes_\mathbb{Z}\mathbb{R}}(L\otimes_{\mathbb{Z}}\mathbb{R})$, where the map $G'\rightarrow GU$ is a closed immersion.

As in section  \ref{unitarygroups-section}, let $\Levin$ be $\prod\limits_{\tau \in \arch} \GL_{\siga_{\tilde\tau}} \times \GL_{\sigb_{\tilde\tau}}$, which can be identified with a Levi subgroup of $U$ over $\bZ_p$. Similarly, we define $\Levin'$ to be $\prod\limits_{\tau \in \arch, 1 \leq i \leq s} \GL_{\siga_{\tilde\tau,i}} \times \GL_{\sigb_{\tilde\tau,i}}$ corresponding to the partition $\left(\siga_{\tau,i}, \sigb_{\tau,i}\right)_{i=1,\dots ,s}$ of the signature $(a_{+\tau},a_{-\tau})$, which can be identified over $\bZ_p$ with a Levi subgroup of  $G' \cap U$.  Note that we have a closed immersion $\Levin' \ra \Levin$ that comes from the natural inclusion of $G'$ in $GU$ over $\bZ_p$. We let $B'$ be the intersection of the Borel $B \subset \Levin$ with $\Levin'$, and (by abuse of notation) we denote by $N$ and $N'$ the $\zz_p$-points of the unipotent radicals of the Borel subgroups $B$ and $B'$, respectively. Then $N'=N\cap \Levin'(\zz_p)$. Furthermore, we may identify the maximal torus $T$ of $\Levin$ with a maximal torus $T'$ in $\Levin'$.  
Note that 
for any character $\kappa$ of $T=T'$, if $\kappa$ is dominant in $ X^*(T)$, then it is also dominant in $X^*(T')$, but the converse is false in general.

The morphism of Shimura data $(G',X')\rightarrow (GU,X)$ described above defines a morphism of Shimura varieties,
$\theta: M'\hookrightarrow M,$
from the Shimura variety $M'$ associated with $G'$ into the Shimura variety $M$ associated with $GU$, which is a closed immersion since $G'\rightarrow GU$ is a closed immersion (\cite[Thm 1.15]{deligne}).  The morphism $\theta$ extends to a map between the canonical integral models (which with abuse of notation we still denote by $\theta$)
\[
\theta: \cM'\hookrightarrow \cM.
\]
A point $x$ in $\cM$ is in the image of $\theta$ if and only if the corresponding abelian variety $\underline{\cA}_x$ decomposes as a cartesian product of abelian varieties, with dimensions and additional structures prescribed by the above data.

\begin{remark}
We describe an example.  Let $K=F$ be a quadratic imaginary field, $V_n$ be an $n$-dimensional $\cO_F$-lattice equipped with a Hermitian pairing $\langle, \rangle$, and assume that the associated group $GU_n=GU(V_n,\langle, \rangle)$ has real signature $(1,n-1)$.
 We fix the partition $\{(1,n-2), (0,1)\}$ of the signature $(1,n-1)$, and we realize $V_{n-1}$ as a direct summand of $V_n$.  We choose a neat level $\Gamma$ hyperspecial at  $p$ and denote by $\Sh_n$ the simple Shimura variety of level $\Gamma$ associated with $GU_n$.  $\Sh_n$ is a classifying space for polarized abelian varieties of dimension $n$, equipped with a compatible action of $\cO_F$.  We write $\cA_n$ for the universal abelian scheme on $\Sh_n$. Then, for each elliptic 
curve $E_0/\zz_p$ with complex multiplication by $\cO_{{F}}$  (corresponding to the choice of a connected component of the $0$-dimensional Shimura variety associated with $GU(0,1)$), the morphism $\theta$ is defined by $\theta^*\cA_n=\cA_{n-1}\times E_0$, and its image $\theta(\Sh_{n-1})$ is a divisor in $\Sh_n$. 
\end{remark}

We now assume that $p$ splits completely in all the reflex fields $E_i$ (and thus also in $E$), where $E_i$ is the reflex field for the integral model $\cM_i$ associated with the group $GU_i$. Then the ordinary loci $\cM^{\rm ord}_i$, $\cM^{' \rm ord}$ and $\Mord$ are non-empty. In particular, a split abelian variety $A=\prod_iA_i$  is ordinary if and only if each of its constituents $A_i$ are ordinary.

Each connected component $\cS'$ of $\cM'$ can be identified with a product of connected components $\cS_i$ of $\cM_i$.  We choose a connected component $\cS'$ of $\cM'$, and identify $\cS'=\prod_{i=1}^s \cS_i$. Then, there is a unique connected component $\cS $ of $\cM$ such that $\theta(\cS')\subset \cS$.
We write $\cS^{\rm ord}_i$ (resp. ${\Sordprime}$ and $ \Sord$) for the ordinary locus of $\cS_i$ (resp. $\cS'$ and $\cS$). Thus,   $\theta(\Sordprime)\subset \Sord$, and we may identify $\Sordprime=\prod_i\cS^{\rm ord}_i$.
Corresponding to our choice of connected components, there are two $\cO_{F,p}$-linear decompositions $\cL^+=\oplus_{i=1}^s\fL^+_i$ and $\cL^-=\oplus_{i=1}^s\fL^-_i$, the ranks of the summands determined by the partition $\left(\siga_i=\sum_{\tau \in \Sigma}\siga_{\wt \tau,i},\sigb_i=\sum_{\tau \in \Sigma}\sigb_{\wt \tau,i}\right)_{i=1,\dots ,s}$ of the signature $(a_+=\sum_{\tau \in \Sigma}\siga_{\wt \tau},a_-=\sum_{\tau \in \Sigma}\sigb_{\wt \tau})$. For each $i=1,\dots, s$, we write $\fL_i=\fL_i^+\oplus \fL_i^-$. Thus, $\cL=\oplus_{i=1}^s\fL_i$.

\begin{prop}

For every level $n\geq 1$, the homomorphism $\theta:\Sordprime=\prod_i\cS^{\rm ord}_i\to\Sord$ lifts canonically to a compatible system of homomorphisms $\Theta=(\Theta_{n})_n$, among the Igusa towers,
\[\Theta_n: \mathrm{Ig}_n^{' \mathrm{ord}}:=\prod_i\Igusa_{n,i} \rightarrow \Igusa_n,\]
where $\Igusa_{n,i}$ denotes the $n$-th level of the Igusa tower over $\Sord_i$,  for each $i=1, \dots , s$.

\end{prop}

\begin{proof}  For each $i=1,\dots ,s$, let $\underline{\cA}_i$ denote the universal abelian scheme over  $\Sord_i$, and  $\iota_i:\mu_{p^n}\otimes\fL_i \hookrightarrow \cA_i[p^n]$ denote the universal Igusa structure of level $n$ on $\cA_i$. By the universal property 
 of $\Igusa_n$ {(i.e. using the fact that the Igusa tower represents the functor classifying ordinary points with additional structure)}, constructing a morphism $\Theta_n$ lifting $\theta$ is equivalent to defining an Igusa structure of level $n$ on  the abelian scheme $\theta^*\underline{\cA}=\prod_i\underline{\cA}_i$ over $\Sordprime$. We define $\iota:\mu_{p^n}\otimes\cL \hookrightarrow \theta^*\cA[p^n]$ as $\iota=\oplus_{i=1}^s \iota_i$.
\end{proof}

\begin{remark}
For each integer $n\geq 0$, the morphism $\Theta_n$ defines a closed embedding of the Igusa covers over $\Sordprime$, \[\mathrm{Ig}_n^{' \mathrm{ord}}\hookrightarrow \Sordprime\times_{\Sord}\Igusa_n.\]   
Indeed, since both projections  $\mathrm{Ig}_n^{' \mathrm{ord}}\to\Sordprime$ and $\Sordprime\times_{\Sord}\Igusa_n \to \Sordprime$ are finite and \'etale, it suffices to check that the map  is one-to-one on points.
Given a point $x_0$ of $\Sordprime$, corresponding to a  split ordinary abelian variaty  $\underline{A}=\prod_i \underline{A}_i$. A point $x$ of  $\Igusa_n$, lying above $x_0$, is in the image of $\Theta_n$ if and only if the corresponding 
Igusa structure $\iota_x:\mu_{p^n}\otimes\cL \hookrightarrow A[p^n]$  satisfies the conditions $\iota_x(\mu_{p^n}\otimes\fL_i)\subset A_i[p^n]$, for all $i=1, \dots , s$. Then, $\iota_x=\oplus_i \iota_{x,i}$, for $\iota_{x,i}:\mu_{p^n}\otimes\fL_i \to A_i[p^n]$ the restrictions of $\iota_x$, i.e. there is a unique point $y\in \mathrm{Ig}_n^{' \mathrm{ord}}$ such that $\Theta_n(y)=x$.

We observe that,  for non-trivial partitions of the signature $(a_{+\tau},a_{-\tau})_{\tau \in \archK}$,  such a closed immersion is not an isomorphim.  In fact, given a point $x$  in $\Theta_n( \mathrm{Ig}_n^{' \mathrm{ord}})\subset\Sordprime\times_{\Sord}\Igusa_n$,
for any $g\in \Levin(\zz_p)$, the point $x^g$ is in the image of $\Theta_n$ if and only if $g\in \Levin'(\zz_p)$. 

\end{remark}

\subsection{Restriction of automorphic forms} 
For $\kappa$ a dominant character of $T$, let $\rho_\kappa$ denote the irreducible representation of $\Levin_{\bZ_p}$ with highest weight $\kappa$ described in Section \ref{Section-weights}. 
Similarly, for $\kappa'$ a dominant character of $T'$, let $\rho'_{\kappa'}$ denote the irreducible representation of  $\Levin'_{\bZ_p}$ with highest weight $\kappa'$.

Assume that the restriction of $\rho_\kappa$ from $\Levin_{\bZ_p}$ to $\Levin'_{\bZ_p}$ has an irreducible {quotient} isomorphic to $\rho'_{\kappa'}$ (e.g., when $\kappa'=\kappa$), and fix a projection $\pi_{\kappa,\kappa'}:\rho_\kappa\to\rho'_{\kappa'}$. If $\kappa'=\kappa ^\sigma$ for some $\sigma$ in the Weyl group $W_\Levin(T)$, then we choose $\pi_{\kappa,\kappa'}$ to satisfy the equality $\ell_{\rm can}^\kappa=\ell_{\rm can}^{\kappa'}\circ \pi_{\kappa,\kappa'}\circ g_\sigma$, where $g_\sigma\in N_\Levin(T)(\zz_p)$ is a lifting of $\sigma$, and where we view $\ell_{\rm can}^\kappa$ as a functional on the space of $\rho_\kappa$ by identifying the space of $\rho_\kappa$ with the space of $\Ind_B^H(\kappa)$ on which $H$ acts by the usual left translation action (described in Section \ref{Section-weights}) precomposed with transpose-inverse.
In other words, $\ell_{\rm can}^\kappa$ is the unique (up to multiple) $N^-$-invariant functional on $\rho_\kappa$.

Recall that $\cE_\kappa=\cE_{\cpct,\kappa}$ denotes the automorphic sheaf of weight $\kappa$ over $\cM$, as defined in Section \ref{classical-aut-forms}, and denote by $\cE'_{\kappa'}$ the automorphic sheaf of weight $\kappa'$ over $\cM'$. Then, on $\cM'$ we have a canonical morphism of sheaves $r_{\kappa,\kappa'}: \theta^*\cE_\kappa\to \cE'_{\kappa'}$. 

\begin{defn}
We define
\[{\rm res}_{\kappa,\kappa'}:= r_{\kappa,\kappa'} \circ\theta^* :  H^0(\cM,\cE_\kappa) \to H^0(\cM',\theta^*\cE_\kappa) \to H^0(\cM',\cE'_{\kappa'}).\]
We call ${\rm res}_{\kappa,\kappa '}$  the {\it weight $(\kappa,\kappa ')$-restriction}.  
\end{defn}

In the following, for $\kappa'=\kappa$, we write ${\rm res}_\kappa={\rm res}_{\kappa,\kappa '}$ and call it the {\em  weight $\kappa$-restriction}.

By abuse of notation we still denote by ${\rm res}_{\kappa,\kappa'}$ (and ${\rm res}_\kappa$) the restriction of this map to the space of sections over the ordinary loci; i.e.
\[{\rm res}_{\kappa,\kappa' }:= r_{\kappa,\kappa'} \circ \theta^*: H^0(\Sord,\cE_\kappa) \to H^0(\Sordprime,\theta^*\cE_\kappa) \to H^0(\Sordprime,\cE'_{\kappa'}).\]

Let $V$, $V'$ denote the spaces of global functions of the Igusa towers $\Igusa/\Sord $ and ${\Igusa}'/\Sordprime$, respectively, as introduced in Section \ref{padicaut-section}.  We write $\Theta^*$ for the pullback on global functions of the Igusa towers, \[\Theta^*:V\to{V'}.\] 
In the following, we refer to $\Theta^*$ as {\it restriction}.
Note that $\Theta $ maps $V^N$ and $V^N[\kappa]$, for any weight $\kappa$, to ${V'}^{N'}$ and ${V'}^{N'}[\kappa]$, respectively. With abuse of notation, we will still denote by $\Theta^*$ the restrictions of $\Theta^*$  to $V^N$ and $V^N[\kappa]$.

Finally, for any $\kappa,\kappa'$, we write $\Psi_\kappa: H^0(\Sord,\cE_\kappa)\to V^N$ and  $\Psi'_{\kappa'}: H^0(\Sordprime,\cE'_{\kappa})\to {V'}^{N'} $ for the inclusions defined in Section \ref{lcan}.  

\begin{prop}\label{propres}
For  $\kappa$ a dominant character of $T$, and any $f\in  H^0(\Sord,\cE_\kappa)$: 
$$\Theta^*(\Psi_\kappa (f))=\Psi'_\kappa({\rm res}_\kappa (f)). $$
\end{prop}
\begin{proof} 
The statement follows from the equality 
$\Theta^*\circ j^*=j^*\circ \theta^*$, together with the observation that, for any dominant weight $\kappa$ of $T$, the 
 functional $\ell_{\rm can}$ appearing in the definitions of $\Psi_\kappa$ (in Section \ref{lcan}) factors by our choice via the projection $\pi_{\kappa,\kappa}: \rho_\kappa\rightarrow \rho'_\kappa$. 
\end{proof}

 Note that if $\kappa'\neq \kappa$, then the maps
$\Theta^*\circ  \Psi_\kappa $ and $\Psi'_{\kappa'} \circ {\rm res}_{\kappa,\kappa'} $ do not agree, as a consequence of Proposition \ref{propres} and the injectivity of the $\Psi$ (see Proposition \ref{ten}). Instead, we have the following result.

\begin{prop} \label{propres2} The notation is as above.
Assume the weight $\kappa'$ is conjugate to $\kappa$ under the action of the Weil group $W_\Levin(T)$, i.e. $\kappa'=\kappa ^\sigma$ for some $\sigma \in W_\Levin(T)$, and choose $g_\sigma\in N_\Levin(T)(\zz_p)$ lifting $\sigma.$ 

Then, for all $f\in H^0(\Sord,\cE_\kappa)$, we have 
\[\Theta^*(g_\sigma\cdot \Psi_\kappa (f))=\Psi'_{\kappa'} ( {\rm res}_{\kappa,\kappa'}(f)) .\]
\end{prop}
\begin{proof}
The same argument as in the proof of Proposition \ref{propres} applies here, because we chose $\pi_{\kappa,\kappa'}$ such that
the 
functional $\ell_{\rm can}$ appearing in the definition of $\Psi_\kappa $ factors via the map $\pi_{\kappa,\kappa'} \circ g_\sigma : \rho_\kappa \to\rho_\kappa \to \rho'_{\kappa '}$. 
\end{proof}

Our goal is to give a simple description of $\Theta^*$ in \ST coordinates, and deduce an explicit criterion for the vanishing of the restriction of a $p$-adic automorphic form in terms of vanishing of some of the coefficients in its $t$-expansion.

Let $x_0\in\Sordprime (\Witt)$, $x_0=(x_0^i)_{i=1, \dots , s}$ where  $x_0^i\in \Sord_i (\Witt)$, for each $i=1, \dots ,s$. We write $\bar{x}_0$ and  $\theta(\bar{x}_0)$ for the reductions modulo $p$ of $x_0$ and of $\theta(x_0)\in\Sord (\Witt)$, respectively. Let $\underline{A}=\underline{\cA}_{\theta(\bar{x}_0)}=\underline{\cA}_{\bar{x}_0}=\prod_i\underline{A}_i$ be the corresponding split ordinary abelian variety over $\fpb$. We deduce that the physical Tate module of $A$ decomposes as \[T_pA(\fpb)=\oplus_i T_p A_i (\fpb),\] 
and by linearity we also have  $T_{\fP_j}A(\fpb)=\oplus_i T_{\fP_j} A_i (\fpb)$,
for each $j=1, \dots , r$.

\begin{prop}
\label{STinclusion} The notation is the same as above.
Under the isomorphism in Proposition \ref{STthm},  $x\mapsto q_x$,  the map
$ \theta_{\bar{x}_0}: {\Sordprime}^\wedge_{\bar{x}_0}\to {\Sord}_{\theta(\bar{x}_0)}^\wedge$ is the closed immersion corresponding to the collection of the natural inclusions
\[\oplus_{i=1}^s{\rm Hom}_{\zz_p}(T_{\fP_j}A_i(\fpb)\otimes T_{\fP_j^c}A_i(\fpb),\hat{\bG}_m)\subset {\rm Hom}_{\zz_p}(T_{\fP_j}A(\fpb)\otimes T_{\fP_j^c}A(\fpb),\hat{\bG}_m),\] 
for $j=1,\ldots , r$. 
\end{prop}

\begin{proof}
By the definition of $\theta$, a point $x\in {\Sord}_{\theta(\bar{x}_0)}^\wedge$ is in the image of $\theta_{\bar{x}_0}$ if and only if the corresponding abelian variety $\underline{\cA}_x$ decomposes as a cartesian product of abelian varieties with additional structures, compatibly with the decomposition $\underline{A}=\prod_i\underline{A}_i $.  We argue that  such a decomposition exists if and only if the endomorphisms $e_i:A\rightarrow A_i\hookrightarrow A$ lift to $\underline{\cA}_x$.

Clearly, if the decomposition lifts to $\ul\cA_x$ so do the endomorphisms $e_i$ for all $i=1, \dots , s$. Vice versa, let us assume there exist endomorphism $\tilde{e}_i$  of $\cA_x$ lifting the $e_i$. By Theorem \ref{thmcon} the endomorphisms $\tilde{e}_i$ are unique, thus in particular they are orthogonal  idempotents (since the $e_i$ are) and  the identity of $\cA_x$ decomposes as $1_{\cA_x}=\sum_i \tilde{e}_i$ (lifting the equality $1_A=\sum_i e_i$). We deduce that for each $i$, the image $\cA_i=\tilde{e}_i(\cA_x)$ is an abelian subvariety of $\cA_x$ lifting $A_i$, and $\cA_x=\prod \cA_i$. Furthermore, for each $i$, the additional structures on $\cA_x$ define unique additional structures on $\cA_i$ (by the properties of the cartesian product) 
which lift those on $A_i$.  To conclude, we observe that since such lifts are unique, the decomposition of $\cA_x$ is compatible with the additional structures, i.e. $\underline{\cA}_x=\prod_i \underline{\cA_i}$ lifting the decomposition of $\underline{A}$. Furthermore,  such lifting is unique.

Finally, by Theorem \ref{thmcon}, for each $i=1,\dots , s$, the endomorphism $e_i$ of $A$ lifts to $\cA_x$ if and only if $q_{\cA_x}\circ (1\times e_i^\vee)=q_{\cA_x}\circ (e_i\times 1)$.
Equivalently, if and only if \[q_x \circ (1\times e_i)=q_x\circ (e_i\times 1)\] (recall $q_x=q_{\cA_x}\circ (1\times T_p(\lambda))$ and under our assumption $e_i^\vee\circ \lambda=\lambda\circ e_i$). We deduce that this is the case if and only if for all $j=1, \dots ,r $ and any $i, k=1,\dots , s$, the restriction of the bilinear form $q_x$ to the subspaces $T_{\fP_j}A_i(\fpb) \otimes T_{\fP_j^c}A_k(\fpb)$\ of $ T_{\fP_j}A(\fpb)\otimes T_{\fP_j^c}A(\fpb)$ vanishes unless $i=k$.
\end{proof}

For each $i=1,\dots ,s$, we define $\cO_{K, p}$-modules
$\fL^2_i\subset \fL_i^{\otimes 2}$ similarly to $\cL^2\subset \cL^{\otimes 2}$ in Section \ref{STthy}.
We consider the $\cO_{K,p}$-module $\fL^2=\oplus_{i =1}^s\fL^2_i$. By the definition $\fL^2$ is a direct summand of  $\cL^2$, we write $\epsilon: \fL^2\to \cL^2$ for the natural inclusion.

We choose a point  $x\in {\Igusa}'(\Witt)$, lying above $x_0$, $(x^i)_i=x$.
By Proposition \ref{propbeta}, associated with the point $x$ on the Igusa tower, we have isomorphisms 
\[{\beta}_x:  {\Sordprime}^\wedge_{\bar{x}}\isomto{\hat{\bG}}_m\otimes \fL^2 \text{ and } \beta_{\Theta(x)}:{\Sord}^\wedge_{\theta(\bar{x}_0)}\isomto{\hat{\bG}}_m\otimes \cL^2 .\]
For all $i=1,\dots ,s$, we may also consider the isomorphism $\beta_{x^i}:(\cS^{\rm ord}_i)^\wedge_{\bar{x}_0^i}\isomto{\hat{\bG}}_m\otimes \fL_i^2$. 
Then, ${\beta}_{x}=\oplus_{i=1}^s (\beta_{x^i})_{i=1,\dots, s}$.  As before, we denote respectively by ${\beta}^*_x$ and $\beta^*_{\Theta(x)}$ the corresponding ring homomorphisms.

\begin{prop}\label{above} The notation is the same as above.
 The map $\beta_{\Theta(x)}\circ \theta_{\bar{x}_0} \circ {\beta}_x^{-1} $ agrees with the inclusion ${\mathbb I\otimes\epsilon}:{\hat{\bG}}_m\otimes \fL^2\to {\hat{\bG}}_m\otimes \cL^2$.
\end{prop}
\begin{proof}
The statement follows from Proposition \ref{STinclusion} and Proposition \ref{propbeta} combined.
\end{proof}

Equivalently, in terms of the local \ST coordinates, Proposition \ref{above} states that the two homomorphisms \[\theta_{x_0}^*:\Ring_{\Sord,\theta(x_0)}^\wedge\rightarrow \Ring_{\Sordprime,x_0}^\wedge \text{ and }{\mathbb I}\otimes \epsilon^{\vee}: \Witt[[t]]\otimes (\cL^{2})^\vee \surjects \Witt[[t]]\otimes (\fL^{2})^\vee\]
satisfy the equality ${\mathbb I}\otimes \epsilon^{\vee}={{\beta}^{*}_x }^{-1}\circ\theta^*_{x_0}\circ \beta^*_{\Theta(x)}$. 

\begin{cor}\label{coro}
The notation is the same as above. For any $f\in V$, we have  \[(\Theta^*f)_x (t)=(\mathbb I\otimes \epsilon^{\vee})(f_{\Theta(x)}(t)).\]
\end{cor}
\begin{proof}
By definition $(\Theta^*f)_x(t)={\beta^*_x}^{-1}\circ {j^*_x}^{-1}((\Theta^* f)_x)$, and $f_{\Theta(x)}(t)= {\beta^*_{\Theta(x)}}^{-1}\circ {j^*_{\Theta(x)}}^{-1}(f_{\Theta(x)})$. Also by definition,  $j\circ \Theta=\theta\circ j$ and  $x_0=j(x)$. Then, for all $f\in V$,
\[
(\Theta^*f)_x(t)={\beta^*_x}^{-1}\circ {j^*_x}^{-1}((\Theta^* f)_x)= {\beta^*_x}^{-1}\circ {j^*_x}^{-1}\circ \Theta_x^* (f_{\Theta(x)})= 
{\beta^*_x}^{-1}\circ \theta_{x_0}^*\circ { j^*_{\Theta(x)}}^{-1} (f_{\Theta(x)}),  \]
and 
\[(\mathbb I\otimes \epsilon^{\vee})(f_{\Theta(x)}(t))=
(1\otimes \epsilon^\vee)\circ {\beta^*_{\Theta(x)}}^{-1}\circ {j^*_{\Theta(x)}}^{-1}(f_{\Theta(x)}).
\]
Thus, the equality ${\mathbb I}\otimes \epsilon^{\vee}\circ {\beta^*_{\Theta(x)}}^{-1}={{\beta}^{*}_x }^{-1}\circ\theta^*_{x_0}$ (following Proposition \ref{above}) suffices to conclude. 
\end{proof}

 We observe that the statement in Corollary \ref{coro} is equivariant for the action of $\Levin'(\zz_p)$. More precisely, the following equalities hold.

\begin{lemma}\label{Leviaction4}  
For any $g\in\Levin'(\zz_p)\subset \Levin(\zz_p)$, $x\in {\Igusa}'(\Witt)$, and  $f\in V$, we have 
\[(\Theta^*f)_{x^g}(t)=({\mathbb I}\otimes g^{-1})(\Theta^*(g\cdot f))_{x}(t) \text{ and }
({\mathbb I}\otimes \epsilon^{\vee})(f_{\Theta(x^g)}(t))=({\mathbb I}\otimes \epsilon^{\vee}\circ g^{-1})((g\cdot f)_{\Theta(x)}(t)).\]
\end{lemma}
\begin{proof}
Recall that $\epsilon\circ g=g\circ \epsilon$ and $\Theta\circ g=g\circ \Theta$,  for all $g\in\Levin'(\zz_p)\subset \Levin(\zz_p)$. Thus,  for any $x\in{\Igusa}'(\Witt)$,  the point $x^g$ is another point of ${\Igusa}'$ satisfying $\Theta(x)^g=\Theta(x^g)$, and for all $f\in V$,  we have $\Theta^*(g\cdot f)=g\cdot \Theta^*f$.  The statement then  follows immediately from Proposition \ref{Leviaction3}.
\end{proof}

By the $t$-expansion principle, we deduce the following vanishing criteria.

\begin{corollary}\label{STrestriction} 
Let $x\in {\Igusa}'(\Witt)$, $f\in H^0(\Sord,\cE_\kappa)$, for any $\kappa$. 

The $\kappa$-restriction  ${\rm res}_\kappa (f)$ of $f$ to the subgroup $G'$ vanishes if and only if \[({\mathbb I}\otimes \epsilon^{\vee})(\Psi_\kappa(f)_{\Theta(x)}(t))=0.\]
\end{corollary}

\begin{proof}
By Theorem \ref{STexpprinciple-thm},  ${\rm res}_\kappa (f)$ vanishes if and only if $\Psi'_\kappa ({\rm res}_\kappa (f))_x (t)$ vanishes.  On the other hand,  Proposition \ref{propres} and Corollary \ref{coro} combined imply 
\[\Psi'_\kappa ({\rm res}_\kappa (f))_x (t)=\Theta^*(\Psi_\kappa(f))_x(t)=({\mathbb I}\otimes\epsilon^\vee)( \Psi_\kappa(f)_{\Theta(x)} (t)).\]
\end{proof}

\begin{corollary}\label{STrestriction2}
Let $x\in {\Igusa}'(\Witt)$, 
$f\in H^0(\Sord,\cE_\kappa)$, for any $\kappa$. 
Let $\kappa'\neq \kappa$ be a dominant weight of $T'$.   

Assume $\kappa'=\kappa^\sigma$, for some $\sigma \in W_\Levin (T)$, and choose $g_\sigma\in N_\Levin(T)(\zz_p)$ lifting $\sigma$.

The $(\kappa,\kappa') $-restriction  ${\rm res}_{\kappa,\kappa'} (f)$ of $f$ to the subgroup $G'$ vanishes if and only if \[({\mathbb I}\otimes (\epsilon^{\vee}\circ g_\sigma))(\Psi_\kappa(f)_{\Theta(x)^{g_\sigma}}(t))=0.\]
\end{corollary}
\begin{proof}
Theorem \ref{STexpprinciple-thm} implies that ${\rm res}_{\kappa,\kappa'} (f)$ vanishes if and only if $\Psi'_{\kappa'} ( {\rm res}_{\kappa,\kappa'}(f))_x(t )$ vanishes.
By combining Proposition \ref{propres2} and Corollary \ref{coro}, we have
\[\Psi'_{\kappa'} ( {\rm res}_{\kappa,\kappa'}(f))_x(t )= \Theta^*(g_\sigma \cdot \Psi_\kappa (f))_x(t)=({\mathbb I}\times \epsilon^{\vee})((g_\sigma\cdot \Psi_\kappa (f))_{\Theta(x)}(t)).\]
Finally, Proposition \ref{Leviaction3}
implies
\[(g_\sigma\cdot \Psi_\kappa (f))_{\Theta(x)}(t)
=({\mathbb I}\otimes  g_\sigma)(\Psi_\kappa (f)_{\Theta(x)^{g_\sigma}}(t)).\]
\end{proof}
Note that, in the above corollary, since $\kappa'\neq \kappa$, $g_\sigma\notin \Levin'(\Witt)$ and the point $\Theta(x)^{g_\sigma}$ is not in the image of $\Theta$,  for any $x\in{\Igusa}'(\Witt)$.

\begin{corollary}\label{generalrest} Let $x\in {\Igusa}'(\Witt)$.
Let $f$ be a global function on the Igusa tower $\{\Igusa_n\}{_{n\geq 0}}$, i.e. $f\in V$. The restriction $\Theta^*f$ of $f$ to the subgroup $G'$ vanishes if and only if 
$({\mathbb I}\otimes \epsilon^{\vee})((g\cdot f)_{\Theta(x)}(t))=0$, or equivalently 
$({\mathbb I}\otimes \epsilon^{\vee})(f_{\Theta(x)^g}(t))=0$, for all $g\in T(\zz_p)=T'(\zz_p)$.
\end{corollary}
\begin{proof}
By Proposition \ref{generalt}, $\Theta^*f$ vanishes if and only if $(g\cdot \Theta^*f)_x(t)$ vanish for all $g\in T(\zz_p)$.  From Lemma \ref{Leviaction4} and Corollary \ref{coro}, for all $g\in T(\zz_p)$ (recall $T(\zz_p)\subset \Levin'(\zz_p)$), we deduce
\[(g\cdot \Theta^*f)_x(t)= ({\mathbb I}\otimes g)(\Theta^*f)_{x^g} (t)= ({\mathbb I}\otimes (g\circ \epsilon^\vee))(f_{\Theta(x^g)}(t))=  ({\mathbb I}\otimes (g\circ \epsilon^\vee))(f_{\Theta(x)^g}(t))\]
On the other hand, we also have 
\[(g\cdot \Theta^*f)_x(t)=(\Theta^*(g\cdot f))_x(t)=({\mathbb I}\otimes\epsilon^\vee)(g\cdot f)_{\Theta(x)}(t).\]
\end{proof}

\begin{corollary} Let $x\in {\Igusa}'(\Witt)$.
Let $r\in \bN$, and let $f,f'$ be two $p$-adic automorphic forms on $GU$ of weights $\kappa,\kappa'$, for any $\kappa, \kappa'$, i.e. $f \in V^N[\kappa], f' \in V^N[\kappa']$. We denote by ${\rm res}_\kappa(f) $ and ${\rm res}_{\kappa'}(f')$ their restrictions to  the subgroup $G'$. Then ${\rm res}_\kappa(f) \equiv {\rm res}_{\kappa'}(f')\mod p^r$ if and only if 
\[\kappa(g)({\mathbb I}\otimes \epsilon^{\vee})(f_{\Theta(x)}(t))\equiv \kappa' (g)({\mathbb I}\otimes \epsilon^{\vee})(f'_{\Theta(x)}(t))\, \mod p^r,\]
for all $g\in T(\zz_p).$
\end{corollary}
\begin{proof}
The statement is an immediate consequence of the Corollary \ref{generalrest}.
\end{proof}

\section{Acknowledgements}
We are grateful to L. Long, R. Pries, and K. Stange for organizing the Women in Numbers 3 workshop and facilitating this collaboration.  We would like to thank the referee for carefully reading the paper and providing many helpful comments, including suggestions for how to improve the introduction.  We would also like to thank M. Harris, H. Hida, and K.-W. Lan for answering questions about $q$-expansion principles.  We are grateful to the Banff International Research Station for creating an ideal working environment. 
\bibliography{WIN3bib}

\end{document}